\documentclass[10pt]{article}
\usepackage{amsmath}
\usepackage{amssymb}
\usepackage{amsthm}
\usepackage{mathrsfs}
\numberwithin{equation}{section}

\begin{document}

\title{Estimates for fractional integral operators and linear commutators on certain weighted amalgam spaces}
\author{Hua Wang \footnote{E-mail address: wanghua@pku.edu.cn.}\\
\footnotesize{College of Mathematics and Econometrics, Hunan University, Changsha 410082, P. R. China}\\
\footnotesize{\&~Department of Mathematics and Statistics, Memorial University, St. John's, NL A1C 5S7, Canada}}
\date{}
\maketitle

\begin{abstract}
In this paper, we first introduce some new classes of weighted amalgam spaces. Then we give the weighted strong-type and weak-type estimates for fractional integral operators $I_\gamma$ on these new function spaces. Furthermore, the weighted strong-type estimate and endpoint estimate of linear commutators $[b,I_{\gamma}]$ generated by $b$ and $I_{\gamma}$ are established as well. In addition, we are going to study related problems about two-weight, weak type inequalities for $I_{\gamma}$ and $[b,I_{\gamma}]$ on the weighted amalgam spaces and give some results. Based on these results and pointwise domination, we can prove norm inequalities involving fractional maximal operator $M_{\gamma}$ and generalized fractional integrals $\mathcal L^{-\gamma/2}$ in the context of weighted amalgam spaces, where $0<\gamma<n$ and $\mathcal L$ is the infinitesimal generator of an analytic semigroup on $L^2(\mathbb R^n)$ with Gaussian kernel bounds.
\end{abstract}

\section{Introduction}

One of the most significant operators in harmonic analysis is the fractional integral operator. Let $n$ be a positive integer. The $n$-dimensional Euclidean space $\mathbb R^n$ is endowed with the Lebesgue measure $dx$ and the Euclidean norm $|\cdot|$. For given $\gamma$, $0<\gamma<n$, the fractional integral operator (or Riesz potential) $I_{\gamma}$ of order $\gamma$ is defined by
\begin{equation}\label{frac}
I_{\gamma}f(x):=\frac{1}{\zeta(\gamma)}\int_{\mathbb R^n}\frac{f(y)}{|x-y|^{n-\gamma}}\,dy,
\quad\mbox{and}\quad \zeta(\gamma)=\frac{\pi^{\frac{n}{\,2\,}}2^\gamma\Gamma(\frac{\gamma}{\,2\,})}{\Gamma(\frac{n-\gamma}{2})}.
\end{equation}
The boundedness properties of $I_{\gamma}$ between various function spaces have been studied extensively. It is well-known that the Hardy--Littlewood--Sobolev theorem states that the fractional integral operator $I_{\gamma}$ is bounded from $L^p(\mathbb R^n)$ to $L^q(\mathbb R^n)$ for $0<\gamma<n$, $1<p<n/{\gamma}$ and $1/q=1/p-{\gamma}/n$. Also we know that $I_{\gamma}$ is bounded from $L^1(\mathbb R^n)$ to $WL^q(\mathbb R^n)$ for $0<\gamma<n$ and $q=n/{(n-\gamma)}$ (see \cite{stein}). In 1974, Muckenhoupt and Wheeden \cite{muckenhoupt} studied the weighted boundedness of $I_{\gamma}$ and obtained the following two results (for sharp weighted norm inequalities, see \cite{lacey}).

\newtheorem{theorem}{Theorem}[section]
\newtheorem{defn}{Definition}[section]
\newtheorem{corollary}{Corollary}[section]
\newtheorem{lemma}{Lemma}[section]

\begin{theorem}[\cite{muckenhoupt}]\label{strong}
Let $0<\gamma<n$, $1<p<n/{\gamma}$, $1/q=1/p-{\gamma}/n$ and $w\in A_{p,q}$. Then the fractional integral operator $I_{\gamma}$ is bounded from $L^p(w^p)$ to $L^q(w^q)$.
\end{theorem}

\begin{theorem}[\cite{muckenhoupt}]\label{weak}
Let $0<\gamma<n$, $p=1$, $q=n/{(n-\gamma)}$ and $w\in A_{1,q}$. Then the fractional integral operator $I_{\gamma}$ is bounded from $L^1(w)$ to $WL^q(w^q)$.
\end{theorem}

For $0<\gamma<n$, the linear commutator $[b,I_{\gamma}]$ generated by a suitable function $b$ and $I_{\gamma}$ is defined by
\begin{align}\label{lfrac}
[b,I_{\gamma}]f(x)&:=b(x)\cdot I_{\gamma}f(x)-I_\gamma(bf)(x)\notag\\
&=\frac{1}{\zeta(\gamma)}\int_{\mathbb R^n}\frac{[b(x)-b(y)]\cdot f(y)}{|x-y|^{n-\gamma}}\,dy.
\end{align}

This commutator was first introduced by Chanillo in \cite{chanillo}. In 1991, Segovia and Torrea \cite{segovia} showed that $[b,I_{\gamma}]$ is bounded from $L^p(w^p)$ ($1<p<n/{\gamma}$) to $L^q(w^q)$ whenever $b\in BMO(\mathbb R^n)$ (see \cite{cruz6} for sharp weighted bounds, see also \cite{chanillo} for the unweighted case).
This corresponds to the norm inequalities satisfied by $I_{\gamma}$. Let us recall the definition of the space of $BMO(\mathbb R^n)$ (see \cite{john}). $BMO(\mathbb R^n)$ is the Banach function space modulo constants with the norm $\|\cdot\|_*$ defined by
\begin{equation*}
\|b\|_*:=\sup_{B:ball}\frac{1}{|B|}\int_B|b(x)-b_B|\,dx<\infty,
\end{equation*}
where the supremum is taken over all balls $B$ in $\mathbb R^n$ and $b_B$ stands for the mean value of $b$ over $B$; that is, $b_B:=\frac{1}{|B|}\int_B b(y)\,dy.$

\begin{theorem}[\cite{segovia}]\label{cstrong}
Let $0<\gamma<n$, $1<p<n/{\gamma}$, $1/q=1/p-{\gamma}/n$ and $w\in A_{p,q}$. Suppose that $b\in BMO(\mathbb R^n)$, then the linear commutator $[b,I_{\gamma}]$ is bounded from $L^p(w^p)$ to $L^q(w^q)$.
\end{theorem}

In the endpoint case $p=1$ and $q=n/{(n-\gamma)}$, since linear commutator $[b,I_{\gamma}]$ has a greater degree of singularity than $I_{\gamma}$ itself, a straightforward computation shows that $[b,I_{\gamma}]$ fails to be of weak type $(1,n/{(n-\gamma)})$ when $b\in BMO(\mathbb R^n)$ (see \cite{cruz4} for some counter-examples). However, if we restrict ourselves to a bounded domain $\Omega$ in $\mathbb R^n$, then the following weighted endpoint estimate for commutator $[b,I_{\gamma}]$ of the fractional integral operator is valid, which was established by Cruz-Uribe and Fiorenza \cite{cruz5} in 2007 (see also \cite{cruz4} for the unweighted case).

\begin{theorem}[\cite{cruz5}]\label{cweak}
Let $0<\gamma<n$, $p=1$, $q=n/{(n-\gamma)}$ and $w^q\in A_1$. Suppose that $b\in BMO(\mathbb R^n)$, then for any given $\lambda>0$ and any bounded domain $\Omega$ in $\mathbb R^n$, there exists a constant $C>0$, which does not depend on $f$, $\Omega$ and $\lambda>0$, such that
\begin{equation*}
\Big[w^q\big(\big\{x\in\Omega:\big|[b,I_{\gamma}]f(x)\big|>\lambda\big\}\big)\Big]^{1/q}
\leq C\int_{\Omega}\Phi\bigg(\frac{|f(x)|}{\lambda}\bigg)\cdot w(x)\,dx,
\end{equation*}
where $\Phi(t)=t\cdot(1+\log^+t)$ and $\log^+t=\max\{\log t,0\}$.
\end{theorem}

Let $1\leq p,s\leq\infty$, a function $f\in L^p_{loc}(\mathbb R^n)$ is said to be in the Wiener amalgam space $(L^p,L^s)(\mathbb R^n)$ of $L^p(\mathbb R^n)$ and $L^s(\mathbb R^n)$, if the function $y\mapsto\|f(\cdot)\cdot\chi_{B(y,1)}(\cdot)\|_{L^p(\mathbb R^n)}$ belongs to $L^s(\mathbb R^n)$, where $B(y,r)=\{x\in\mathbb R^n:|x-y|<r\}$ is the open ball centered at $y$ and with radius $r$, $\chi_{B(y,r)}$ is the characteristic function of the ball $B(y,r)$, and $\|\cdot\|_{L^p}$ is the usual Lebesgue norm in $L^p(\mathbb R^n)$. Define
\begin{equation}\label{wiener}
(L^p,L^s)(\mathbb R^n):=\left\{f:\big\|f\big\|_{(L^p,L^s)(\mathbb R^n)}
=\left(\int_{\mathbb R^n}\Big[\big\|f\cdot\chi_{B(y,1)}\big\|_{L^p(\mathbb R^n)}\Big]^sdy\right)^{1/s}<\infty\right\}.
\end{equation}
Then we know that $(L^p,L^s)(\mathbb R^n)$ becomes a Banach function space with respect to the norm $\|\cdot\|_{(L^p,L^s)(\mathbb R^n)}$. This amalgam space was first introduced by Wiener in the 1920's, but its systematic study goes back to the works of Holland \cite{holland}, Fournier and Stewart \cite{F}. Let $1\leq p,s,\alpha\leq\infty$. We define the amalgam space $(L^p,L^s)^{\alpha}(\mathbb R^n)$ of $L^p(\mathbb R^n)$ and $L^s(\mathbb R^n)$ as the set of all measurable functions $f$ satisfying $f\in L^p_{loc}(\mathbb R^n)$ and $\big\|f\big\|_{(L^p,L^s)^{\alpha}(\mathbb R^n)}<\infty$, where
\begin{equation*}
\begin{split}
\big\|f\big\|_{(L^p,L^s)^{\alpha}(\mathbb R^n)}
:=&\sup_{r>0}\left\{\int_{\mathbb R^n}\Big[\big|B(y,r)\big|^{1/{\alpha}-1/p-1/s}\big\|f\cdot\chi_{B(y,r)}\big\|_{L^p(\mathbb R^n)}\Big]^sdy\right\}^{1/s}\\
=&\sup_{r>0}\Big\|\big|B(y,r)\big|^{1/{\alpha}-1/p-1/s}\big\|f\cdot\chi_{B(y,r)}\big\|_{L^p(\mathbb R^n)}\Big\|_{L^s(\mathbb R^n)},
\end{split}
\end{equation*}
with the usual modification when $p=\infty$ or $s=\infty$ and $|B(y,r)|$ is the Lebesgue measure of the ball $B(y,r)$. This generalization of amalgam space was originally introduced by Fofana in \cite{fofana}. As proved in \cite{fofana} the space $(L^p,L^s)^{\alpha}(\mathbb R^n)$ is non-trivial if and only if $p\leq\alpha\leq s$; thus in the remaining of the paper we will always assume that this condition $p\leq\alpha\leq s$ is fulfilled. Note that
\begin{itemize}
  \item For $1\leq p\leq\alpha\leq s\leq\infty$, one can easily see that $(L^p,L^s)^{\alpha}(\mathbb R^n)\subseteq(L^p,L^s)(\mathbb R^n)$, where $(L^p,L^s)(\mathbb R^n)$ is the Wiener amalgam space defined by \eqref{wiener};
  \item if $1\leq p<\alpha$ and $s=\infty$, then $(L^p,L^s)^{\alpha}(\mathbb R^n)$ is just the classical Morrey space $\mathcal L^{p,\kappa}(\mathbb R^n)$ defined by (with $\kappa=1-p/{\alpha}$, see \cite{morrey})
\begin{equation*}
\mathcal L^{p,\kappa}(\mathbb R^n):=\left\{f:\big\|f\big\|_{\mathcal L^{p,\kappa}(\mathbb R^n)}
=\sup_{y\in\mathbb R^n,r>0}\left(\frac{1}{|B(y,r)|^\kappa}\int_{B(y,r)}|f(x)|^p\,dx\right)^{1/p}<\infty\right\};
\end{equation*}
  \item if $p=\alpha$ and $s=\infty$, then $(L^p,L^s)^{\alpha}(\mathbb R^n)$ reduces to the usual Lebesgue space $L^{p}(\mathbb R^n)$.
\end{itemize}

In \cite{feuto2} (see also \cite{feuto1,feuto3}), Feuto considered a weighted version of the amalgam space $(L^p,L^s)^{\alpha}(w)$. A non-negative measurable function $w$ defined on $\mathbb R^n$ is called a weight if it is locally integrable. Let $1\leq p\leq\alpha\leq s\leq\infty$ and $w$ be a weight on $\mathbb R^n$. We denote by $(L^p,L^s)^{\alpha}(w)$ the weighted amalgam space, the space of all locally integrable functions $f$ satisfying $\big\|f\big\|_{(L^p,L^s)^{\alpha}(w)}<\infty$, where
\begin{align}\label{A}
\big\|f\big\|_{(L^p,L^s)^{\alpha}(w)}
:=&\sup_{r>0}\left\{\int_{\mathbb R^n}\Big[w(B(y,r))^{1/{\alpha}-1/p-1/s}\big\|f\cdot\chi_{B(y,r)}\big\|_{L^p(w)}\Big]^sdy\right\}^{1/s}\notag\\
=&\sup_{r>0}\Big\|w(B(y,r))^{1/{\alpha}-1/p-1/s}\big\|f\cdot\chi_{B(y,r)}\big\|_{L^p(w)}\Big\|_{L^s(\mathbb R^n)},
\end{align}
with the usual modification when $s=\infty$ and $w(B(y,r)):=\int_{B(y,r)}w(x)\,dx$ is the weighted measure of $B(y,r)$. Similarly, for $1\leq p\leq\alpha\leq s\leq\infty$, we can see that $(L^p,L^s)^{\alpha}(w)$ becomes a Banach function space with respect to the norm $\|\cdot\|_{(L^p,L^s)^{\alpha}(w)}$. Furthermore, we denote by $(WL^p,L^s)^{\alpha}(w)$ the weighted weak amalgam space consisting of all measurable functions $f$ such that (see \cite{feuto2})
\begin{align}\label{WA}
\big\|f\big\|_{(WL^p,L^s)^{\alpha}(w)}
:=&\sup_{r>0}\left\{\int_{\mathbb R^n}\Big[w(B(y,r))^{1/{\alpha}-1/p-1/s}\big\|f\cdot\chi_{B(y,r)}\big\|_{WL^p(w)}\Big]^sdy\right\}^{1/s}\notag\\
=&\sup_{r>0}\Big\|w(B(y,r))^{1/{\alpha}-1/p-1/s}\big\|f\cdot\chi_{B(y,r)}\big\|_{WL^p(w)}\Big\|_{L^s(\mathbb R^n)}<\infty.
\end{align}

Notice that
\begin{itemize}
  \item If $1\leq p<\alpha$ and $s=\infty$, then $(L^p,L^s)^{\alpha}(w)$ is just the weighted Morrey space $\mathcal L^{p,\kappa}(w)$ defined by (with $\kappa=1-p/{\alpha}$, see \cite{komori})
\begin{equation*}
\begin{split}
&\mathcal L^{p,\kappa}(w)\\
:=&\left\{f :\big\|f\big\|_{\mathcal L^{p,\kappa}(w)}
=\sup_{y\in\mathbb R^n,r>0}\bigg(\frac{1}{w(B(y,r))^{\kappa}}\int_{B(y,r)}|f(x)|^pw(x)\,dx\bigg)^{1/p}<\infty\right\},
\end{split}
\end{equation*}
and $(WL^p,L^s)^{\alpha}(w)$ is just the weighted weak Morrey space $W\mathcal L^{p,\kappa}(w)$ defined by (with $\kappa=1-p/{\alpha}$, see \cite{wang1})
\begin{equation*}
\begin{split}
&W\mathcal L^{p,\kappa}(w)\\
:=&\left\{f :\big\|f\big\|_{W\mathcal L^{p,\kappa}(w)}
=\sup_{y\in\mathbb R^n,r>0}\sup_{\lambda>0}\frac{1}{w(B(y,r))^{\kappa/p}}\lambda\cdot\Big[w\big(\big\{x\in B(y,r):|f(x)|>\lambda\big\}\big)\Big]^{1/p}
<\infty\right\};
\end{split}
\end{equation*}
  \item if $p=\alpha$ and $s=\infty$, then $(L^p,L^s)^{\alpha}(w)$ reduces to the weighted Lebesgue space $L^{p}(w)$, and $(WL^p,L^s)^{\alpha}(w)$ reduces to the weighted weak Lebesgue space $WL^{p}(w)$.
\end{itemize}

Recently, many works in classical harmonic analysis have been devoted to norm inequalities involving several integral operators in the setting of weighted amalgam spaces, see \cite{feuto4,feuto1,feuto2,feuto3} and \cite{wei}. These results obtained are extensions of well-known analogues in the weighted Lebesgue spaces.

Let $I_\gamma$ be the fractional integral operator, and let $[b,I_{\gamma}]$ be its linear commutator. The aim of this paper is twofold. We first define some new classes of weighted amalgam spaces. As the weighted amalgam space may be considered as an extension of the weighted Lebesgue space, it is natural and important to study the weighted boundedness of $I_{\gamma}$ and $[b,I_{\gamma}]$ in these new spaces. We will prove that $I_{\gamma}$ as well as its commutator $[b,I_{\gamma}]$ which are known to be bounded on weighted Lebesgue spaces, are bounded on weighted amalgam spaces under appropriate conditions. In addition, we will discuss two-weight, weak type norm inequalities for $I_\gamma$ and $[b,I_{\gamma}]$ in the context of weighted amalgam spaces and give some partial results. Using these results and pointwise domination, we will establish the corresponding strong-type and weak-type estimates for fractional maximal operator $M_{\gamma}$ and generalized fractional integrals $\mathcal L^{-\gamma/2}$, where $0<\gamma<n$ and $\mathcal L$ is the infinitesimal generator of an analytic semigroup on $L^2(\mathbb R^n)$ with Gaussian kernel bounds.

The present paper is organized as follows. In $\S$2, we first state some preliminary definitions and results about $A_p$ weights, Orlicz spaces and weighted amalgam spaces, and the main results of the present paper are also given in $\S$2. The following $\S$3, $\S$4 and $\S$5 are devoted to their proofs. Finally, in $\S$6 we discuss some related two-weight problems.

\section{Statement of our main results}

\subsection{Notations and preliminaries}

Let us first recall the definitions of two weight classes; $A_p$ and $A_{p,q}$.
\begin{defn}[$A_p$ weights \cite{muckenhoupt1}]
A weight $w$ is said to belong to the class $A_p$ for $1<p<\infty$, if there exists a positive constant $C$ such that for any ball $B$ in $\mathbb R^n$,
\begin{equation*}
\left(\frac1{|B|}\int_B w(x)\,dx\right)^{1/p}\left(\frac1{|B|}\int_B w(x)^{-p'/p}\,dx\right)^{1/{p'}}\leq C<\infty,
\end{equation*}
where we denote the conjugate exponent of $p>1$ by $p'=p/{(p-1)}$. The class $A_1$ is defined replacing the above inequality by
\begin{equation*}
\frac1{|B|}\int_B w(x)\,dx\leq C\cdot\underset{x\in B}{\mbox{ess\,inf}}\;w(x),
\end{equation*}
for any ball $B$ in $\mathbb R^n$. We also define $A_\infty=\bigcup_{1\leq p<\infty}A_p$.
\end{defn}

\begin{defn}[$A_{p,q}$ weights \cite{muckenhoupt}]
A weight $w$ is said to belong to the class $A_{p,q}$ for $1<p,q<\infty$, if there exists a positive constant $C$ such that for any ball $B$ in $\mathbb R^n$,
\begin{equation*}
\left(\frac1{|B|}\int_B w(x)^q\,dx\right)^{1/q}\left(\frac1{|B|}\int_B w(x)^{-p'}\,dx\right)^{1/{p'}}\leq C<\infty.
\end{equation*}
The class $A_{1,q}$ $(1<q<\infty)$ is defined replacing the above inequality by
\begin{equation*}
\left(\frac1{|B|}\int_B w(x)^q\,dx\right)^{1/q}\left(\underset{x\in B}{\mbox{ess\,sup}}\;\frac{1}{w(x)}\right)\leq C<\infty,
\end{equation*}
for any ball $B$ in $\mathbb R^n$.
\end{defn}
There is a close connection between $A_p$ weights and $A_{p,q}$ weights (see \cite{lu}).
\begin{lemma}\label{relation}
Suppose that $0<\gamma<n$, $1\leq p<n/{\gamma}$ and $1/q=1/p-{\gamma}/n$. Then the following statements are true:
\begin{description}
\item $(i)$ If $p>1$, then $w\in A_{p,q}$ implies $w^q\in A_q$ and $w^{-p'}\in A_{p'};$
\item $(ii)$ if $p=1$, then $w\in A_{1,q}$ if and only if $w^q\in A_1$.
\end{description}
\end{lemma}

Given a ball $B$ and $\lambda>0$, we write $\lambda B$ for the ball with the same center as $B$ whose radius is $\lambda$ times that of $B$. For any $r>0$ and $y\in\mathbb R^n$, we denote by $B(y,r)^c$ the complement of $B(y,r)$ in $\mathbb R^n$; that is $B(y,r)^c:=\mathbb R^n\backslash B(y,r)$. Given a weight $w$, we say that $w$ satisfies the \emph{doubling} condition if there exists a universal constant $C>0$ such that for any ball $B$ in $\mathbb R^n$, we have
\begin{equation}\label{weights}
w(2B)\leq C\cdot w(B).
\end{equation}
When $w$ satisfies this \emph{doubling} condition \eqref{weights}, we denote $w\in\Delta_2$ for brevity. An important fact here is that if $w$ is in $A_{\infty}$, then $w\in\Delta_2$ (see \cite{garcia}). Moreover, if $w\in A_\infty$, then there exists a number $\delta>0$ such that (see \cite{garcia})
\begin{equation}\label{compare}
\frac{w(E)}{w(B)}\le C\left(\frac{|E|}{|B|}\right)^\delta
\end{equation}
holds for any measurable subset $E$ of a ball $B$.

Given a weight $w$ on $\mathbb R^n$, for $1\leq p<\infty$, the weighted Lebesgue space $L^p(w)$ is defined as the set of all functions $f$ such that
\begin{equation*}
\big\|f\big\|_{L^p(w)}:=\bigg(\int_{\mathbb R^n}|f(x)|^pw(x)\,dx\bigg)^{1/p}<\infty.
\end{equation*}
We also denote by $WL^p(w)$($1\leq p<\infty$) the weighted weak Lebesgue space consisting of all measurable functions $f$ such that
\begin{equation*}
\big\|f\big\|_{WL^p(w)}:=
\sup_{\lambda>0}\lambda\cdot\Big[w\big(\big\{x\in\mathbb R^n:|f(x)|>\lambda\big\}\big)\Big]^{1/p}<\infty.
\end{equation*}

We next recall some definitions and basic facts about Orlicz spaces needed for the proofs of our main results. For further information on this subject, we refer to \cite{rao}. A function $\mathcal A:[0,+\infty)\rightarrow[0,+\infty)$ is said to be a Young function if it is continuous, convex and strictly increasing satisfying $\mathcal A(0)=0$ and $\mathcal A(t)\to +\infty$ as $t\to +\infty$. An important example of Young function is $\mathcal A(t)=t^p\cdot(1+\log^+t)^p$ with some $1\leq p<\infty$. Given a Young function $\mathcal A$, we define the $\mathcal A$-average of a function $f$ over a ball $B$ by means of the following Luxemburg norm:
\begin{equation*}
\big\|f\big\|_{\mathcal A,B}
:=\inf\left\{\lambda>0:\frac{1}{|B|}\int_B\mathcal A\left(\frac{|f(x)|}{\lambda}\right)dx\leq1\right\}.
\end{equation*}
In particular, when $\mathcal A(t)=t^p$, $1\leq p<\infty$, it is easy to see that $\mathcal A$ is a Young function and
\begin{equation*}
\big\|f\big\|_{\mathcal A,B}=\left(\frac{1}{|B|}\int_B|f(x)|^p\,dx\right)^{1/p};
\end{equation*}
that is, the Luxemburg norm coincides with the normalized $L^p$ norm. Recall that the following generalization of H\"older's inequality holds:
\begin{equation*}
\frac{1}{|B|}\int_B\big|f(x)\cdot g(x)\big|\,dx\leq 2\big\|f\big\|_{\mathcal A,B}\big\|g\big\|_{\bar{\mathcal A},B},
\end{equation*}
where $\bar{\mathcal A}$ is the complementary Young function associated to $\mathcal A$, which is given by $\bar{\mathcal A}(s):=\sup_{0\leq t<\infty}[st-\mathcal A(t)]$, $0\leq s<\infty$. Obviously, $\Phi(t)=t\cdot(1+\log^+t)$ is a Young function and its complementary Young function is $\bar{\Phi}(t)\approx e^t-1$. In the present situation, we denote $\|f\|_{\Phi,B}$ and $\|g\|_{\bar{\Phi},B}$ by $\|f\|_{L\log L,B}$ and $\|g\|_{\exp L,B}$, respectively. Now the above generalized H\"older's inequality reads
\begin{equation}\label{holder}
\frac{1}{|B|}\int_B\big|f(x)\cdot g(x)\big|\,dx\leq 2\big\|f\big\|_{L\log L,B}\big\|g\big\|_{\exp L,B}.
\end{equation}
There is a further generalization of H\"older's inequality that turns out to be useful for our purpose (see \cite{neil}): Let $\mathcal A$, $\mathcal B$ and $\mathcal C$ be Young functions such that for all $t>0$,
\begin{equation*}
\mathcal A^{-1}(t)\cdot\mathcal B^{-1}(t)\leq\mathcal C^{-1}(t),
\end{equation*}
where $\mathcal A^{-1}(t)$ is the inverse function of $\mathcal A(t)$. Then for all functions $f$ and $g$, and for all balls $B$ in $\mathbb R^n$,
\begin{equation}\label{three}
\big\|f\cdot g\big\|_{\mathcal C,B}\leq 2\big\|f\big\|_{\mathcal A,B}\big\|g\big\|_{\mathcal B,B}.
\end{equation}

\subsection{Weighted amalgam spaces}

Let us begin with the definitions of the weighted amalgam spaces with Lebesgue measure in (\ref{A}) and (\ref{WA}) replaced by weighted measure.
\begin{defn}\label{amalgam}
Let $1\leq p\leq\alpha\leq s\leq\infty$, and let $\nu,w,\mu$ be three weights on $\mathbb R^n$. We denote by $(L^p,L^s)^{\alpha}(\nu,w;\mu)$ the weighted amalgam space, the space of all locally integrable functions $f$ such that
\begin{equation*}
\begin{split}
\big\|f\big\|_{(L^p,L^s)^{\alpha}(\nu,w;\mu)}
:=&\sup_{r>0}\left\{\int_{\mathbb R^n}\Big[w(B(y,r))^{1/{\alpha}-1/p-1/s}\big\|f\cdot\chi_{B(y,r)}\big\|_{L^p(\nu)}\Big]^s\mu(y)\,dy\right\}^{1/s}\\
=&\sup_{r>0}\Big\|w(B(y,r))^{1/{\alpha}-1/p-1/s}\big\|f\cdot\chi_{B(y,r)}\big\|_{L^p(\nu)}\Big\|_{L^s(\mu)}<\infty.
\end{split}
\end{equation*}
If $\nu=w$, then we denote $(L^p,L^s)^{\alpha}(w;\mu)$ for brevity, i.e., $(L^p,L^s)^{\alpha}(w,w;\mu):=(L^p,L^s)^{\alpha}(w;\mu)$.
Furthermore, we denote by $(WL^p,L^s)^{\alpha}(w;\mu)$ the weighted weak amalgam space consisting of all measurable functions $f$ for which
\begin{equation*}
\begin{split}
\big\|f\big\|_{(WL^p,L^s)^{\alpha}(w;\mu)}
:=&\sup_{r>0}\left\{\int_{\mathbb R^n}\Big[w(B(y,r))^{1/{\alpha}-1/p-1/s}\big\|f\cdot\chi_{B(y,r)}\big\|_{WL^p(w)}\Big]^s\mu(y)\,dy\right\}^{1/s}\\
=&\sup_{r>0}\Big\|w(B(y,r))^{1/{\alpha}-1/p-1/s}\big\|f\cdot\chi_{B(y,r)}\big\|_{WL^p(w)}\Big\|_{L^s(\mu)}<\infty,
\end{split}
\end{equation*}
with the usual modification when $s=\infty$.
\end{defn}

The aim of this paper is to extend Theorems 1.1--1.4 to the corresponding weighted amalgam spaces. We are going to prove that the fractional integral operator $I_{\gamma}$ which is bounded on weighted Lebesgue spaces, is also bounded on our new weighted spaces under appropriate conditions. Our first two results in this paper is stated as follows.
\begin{theorem}\label{mainthm:1}
Let $0<\gamma<n$, $1<p<n/{\gamma}$, $1/q=1/p-{\gamma}/n$ and $w\in A_{p,q}$. Assume that $p\leq\alpha<\beta<s\leq\infty$ and $\mu\in\Delta_2$, then the fractional integral operator $I_{\gamma}$ is bounded from $(L^p,L^s)^{\alpha}(w^p,w^q;\mu)$ into $(L^q,L^s)^{\beta}(w^q;\mu)$ with $1/{\beta}=1/{\alpha}-{\gamma}/n$.
\end{theorem}

\begin{theorem}\label{mainthm:2}
Let $0<\gamma<n$, $p=1$, $q=n/{(n-\gamma)}$ and $w\in A_{1,q}$. Assume that $1\leq\alpha<\beta<s\leq\infty$ and $\mu\in\Delta_2$, then the fractional integral operator $I_{\gamma}$ is bounded from $(L^1,L^s)^{\alpha}(w,w^q;\mu)$ into $(WL^q,L^s)^{\beta}(w^q;\mu)$ with $1/{\beta}=1/{\alpha}-{\gamma}/n$.
\end{theorem}

Let $[b,I_{\gamma}]$ be the linear commutator generated by $I_{\gamma}$ and $BMO$ function $b$. For the strong-type estimate of $[b,I_{\gamma}]$ on the weighted amalgam spaces, we have the following result:
\begin{theorem}\label{mainthm:3}
Let $0<\gamma<n$, $1<p<n/{\gamma}$, $1/q=1/p-{\gamma}/n$ and $w\in A_{p,q}$. Assume that $p\leq\alpha<\beta<s\leq\infty$, $\mu\in\Delta_2$ and $b\in BMO(\mathbb R^n)$, then the linear commutator $[b,I_{\gamma}]$ is bounded from $(L^p,L^s)^{\alpha}(w^p,w^q;\mu)$ into $(L^q,L^s)^{\beta}(w^q;\mu)$ with $1/{\beta}=1/{\alpha}-{\gamma}/n$.
\end{theorem}

To obtain endpoint estimate for the linear commutator $[b,I_{\gamma}]$, we first need to define the weighted $\mathcal A$-average of a function $f$ over a ball $B$ by means of the weighted Luxemburg norm; that is, given a Young function $\mathcal A$ and $w\in A_\infty$, we define (see \cite{rao} for instance)
\begin{equation*}
\big\|f\big\|_{\mathcal A(w),B}:=\inf\left\{\lambda>0:\frac{1}{w(B)}
\int_B\mathcal A\left(\frac{|f(x)|}{\lambda}\right)\cdot w(x)\,dx\leq1\right\}.
\end{equation*}
When $\mathcal A(t)=t$, this norm is denoted by $\|\cdot\|_{L(w),B}$, and when $\Phi(t)=t\cdot(1+\log^+t)$, this norm is also denoted by $\|\cdot\|_{L\log L(w),B}$. The complementary Young function is given by $\bar{\Phi}(t)\approx e^t-1$ with corresponding mean Luxemburg norm denoted by $\|\cdot\|_{\exp L(w),B}$. For $w\in A_\infty$ and for every ball $B$ in $\mathbb R^n$, we can also show the weighted version of \eqref{holder}. Namely, the following generalized H\"older's inequality in the weighted setting
\begin{equation}\label{Wholder}
\frac{1}{w(B)}\int_B|f(x)\cdot g(x)|w(x)\,dx\leq C\big\|f\big\|_{L\log L(w),B}\big\|g\big\|_{\exp L(w),B}
\end{equation}
is true (see \cite{zhang} for instance). Now we introduce new amalgam spaces of $L\log L$ type as follows.
\begin{defn}
Let $p=1$, $1\leq\alpha\leq s\leq\infty$, and let $\nu,w,\mu$ be three weights on $\mathbb R^n$. We denote by $(L\log L,L^s)^{\alpha}(\nu,w;\mu)$ the weighted amalgam space of $L\log L$ type, the space of all locally integrable functions $f$ defined on $\mathbb R^n$ with finite norm $\big\|f\big\|_{(L\log L,L^s)^{\alpha}(\nu,w;\mu)}$.
\begin{equation*}
(L\log L,L^s)^{\alpha}(\nu,w;\mu):=\left\{f:\big\|f\big\|_{(L\log L,L^s)^{\alpha}(\nu,w;\mu)}<\infty\right\},
\end{equation*}
where
\begin{equation*}
\begin{split}
\big\|f\big\|_{(L\log L,L^s)^{\alpha}(\nu,w;\mu)}
:=&\sup_{r>0}\left\{\int_{\mathbb R^n}
\Big[w(B(y,r))^{1/{\alpha}-1-1/s}\nu(B(y,r))\big\|f\big\|_{L\log L(\nu),B(y,r)}\Big]^s\mu(y)\,dy\right\}^{1/s}\\
=&\sup_{r>0}\Big\|w(B(y,r))^{1/{\alpha}-1-1/s}\nu(B(y,r))\big\|f\big\|_{L\log L(\nu),B(y,r)}\Big\|_{L^s(\mu)}.
\end{split}
\end{equation*}
\end{defn}
Note that $t\leq t\cdot(1+\log^+t)$ for all $t>0$. Then for any ball $B$ in $\mathbb R^n$ and $\nu\in A_\infty$, it is immediate that $\big\|f\big\|_{L(v),B}\leq \big\|f\big\|_{L\log L(v),B}$ by definition, i.e., the inequality
\begin{equation}\label{main esti1}
\big\|f\big\|_{L(\nu),B}=\frac{1}{\nu(B)}\int_B|f(x)|\cdot\nu(x)\,dx\leq\big\|f\big\|_{L\log L(\nu),B}
\end{equation}
holds for any ball $B$ in $\mathbb R^n$. From this, we can further see the following inclusion:
\begin{equation*}
(L\log L,L^s)^{\alpha}(\nu,w;\mu)\subseteq(L^1,L^s)^{\alpha}(\nu,w;\mu),
\end{equation*}
when $1\leq\alpha\leq s\leq\infty$ and $w,\mu$ are some other weights.

In the endpoint case $p=1$, we will prove the following weak-type $L\log L$ estimate of linear commutator $[b,I_{\gamma}]$ in the setting of weighted amalgam spaces.

\begin{theorem}\label{mainthm:4}
Let $0<\gamma<n$, $p=1$, $q=n/{(n-\gamma)}$ and $w\in A_{1,q}$. Assume that $1\leq\alpha<\beta<s\leq\infty$, $\mu\in\Delta_2$ and $b\in BMO(\mathbb R^n)$, then for any given $\lambda>0$ and any ball $B(y,r)$ in $\mathbb R^n$, there exists a constant $C>0$ independent of $f$, $B(y,r)$ and $\lambda>0$ such that
\begin{equation*}
\begin{split}
&\Big\|w^q(B(y,r))^{1/{\beta}-1/q-1/s}\cdot \Big[w^q\big(\big\{x\in B(y,r):\big|[b,I_{\gamma}](f)(x)\big|>\lambda\big\}\big)\Big]^{1/q}\Big\|_{L^s(\mu)}\\
&\leq C\cdot\bigg\|\Phi\left(\frac{|f|}{\,\lambda\,}\right)\bigg\|_{(L\log L,L^s)^{\alpha}(w,w^q;\mu)},
\end{split}
\end{equation*}
where $\Phi(t)=t\cdot(1+\log^+t)$ and $1/{\beta}=1/{\alpha}-{\gamma}/n$. From the above definitions, we can roughly say that the linear commutator $[b,I_{\gamma}]$ is bounded from $(L\log L,L^s)^{\alpha}(w,w^q;\mu)$ into $(WL^q,L^s)^{\beta}(w^q;\mu)$.
\end{theorem}

Moreover, we will discuss the extreme case $\beta=s$ of Theorem \ref{mainthm:1}. In order to do so, we need to introduce new $BMO$-type space given below.
\begin{defn}
Let $1\leq s\leq\infty$ and $\mu\in\Delta_2$. We define the space $(BMO,L^s)(\mu)$ as the set of all locally integrable functions $f$ satisfying $\|f\|_{**}<\infty$, where
\begin{equation}\label{BMOw}
\|f\|_{**}:=\sup_{r>0}\bigg\|\frac{1}{|B(y,r)|}\int_{B(y,r)}\big|f(x)-f_{B(y,r)}\big|\,dx\bigg\|_{L^s(\mu)}.
\end{equation}
Here the $L^s(\mu)$-norm is taken with respect to the variable $y$. We also use the notation $f_{B(y,r)}$ to denote the mean value of $f$ over $B(y,r)$.
\end{defn}

Observe that if $s=\infty$, then $(BMO,L^s)(\mu)$ is the classical $BMO$ space.

Now we can show that $I_{\gamma}$ is bounded from $(L^p,L^s)^{\alpha}(w^p,w^q;\mu)$ into our new $BMO$-type space defined above. This result can be regarded as a supplement of Theorem \ref{mainthm:1}.
\begin{theorem}\label{mainthm:end}
Let $0<\gamma<n$, $1<p<n/{\gamma}$, $1/q=1/p-{\gamma}/n$, and let $w\in A_{p,q}$ and $\mu\in\Delta_2$. If $p\leq\alpha<s\leq\infty$ and $1/{s}=1/{\alpha}-{\gamma}/n$, then the fractional integral operator $I_{\gamma}$ is bounded from $(L^p,L^s)^{\alpha}(w^p,w^q;\mu)$ into $(BMO,L^s)(\mu)$.
\end{theorem}

Throughout this paper, the letter $C$ always denotes a positive constant that is independent of the essential variables but whose value may vary at each occurrence. We also use $A\approx B$ to denote the equivalence of $A$ and $B$; that is, there exist two positive constants $C_1$, $C_2$ independent of quantities $A$ and $B$ such that $C_1 A\leq B\leq C_2 A$. Equivalently, we could define the above notions of this section with cubes in place of balls and we will use whichever is more appropriate, depending on the circumstances.

\section{Proofs of Theorems \ref{mainthm:1} and \ref{mainthm:2}}
In this section, we will prove the conclusions of Theorems \ref{mainthm:1} and \ref{mainthm:2}.
\begin{proof}[Proof of Theorem $\ref{mainthm:1}$]
The proof is inspired by \cite{feuto1,feuto2}. Let $1<p\leq\alpha<s\leq\infty$ and $f\in(L^p,L^s)^{\alpha}(w^p,w^q;\mu)$ with $w\in A_{p,q}$ and $\mu\in\Delta_2$.
For an arbitrary point $y\in\mathbb R^n$, set $B=B(y,r)$ for the ball centered at $y$ and of radius $r$, $2B=B(y,2r)$. We represent $f$ as
\begin{equation*}
f=f\cdot\chi_{2B}+f\cdot\chi_{(2B)^c}:=f_1+f_2;
\end{equation*}
where $\chi_{2B}$ is the characteristic function of $2B$. By the linearity of the fractional integral operator $I_{\gamma}$, one can write
\begin{align}\label{I}
&w^q(B(y,r))^{1/{\beta}-1/q-1/s}\big\|I_\gamma(f)\cdot\chi_{B(y,r)}\big\|_{L^q(w^q)}\notag\\
&=w^q(B(y,r))^{1/{\beta}-1/q-1/s}\bigg(\int_{B(y,r)}\big|I_\gamma(f)(x)\big|^qw^q(x)\,dx\bigg)^{1/q}\notag\\
&\leq w^q(B(y,r))^{1/{\beta}-1/q-1/s}\bigg(\int_{B(y,r)}\big|I_\gamma(f_1)(x)\big|^qw^q(x)\,dx\bigg)^{1/q}\notag\\
&+w^q(B(y,r))^{1/{\beta}-1/q-1/s}\bigg(\int_{B(y,r)}\big|I_\gamma(f_2)(x)\big|^qw^q(x)\,dx\bigg)^{1/q}\notag\\
&:=I_1(y,r)+I_2(y,r).
\end{align}
Here and in what follows, for any positive number $\tau>0$, we use the convention $f^\tau(x):=[f(x)]^{\tau}$. Below we will give the estimates of $I_1(y,r)$ and $I_2(y,r)$, respectively. By the weighted $(L^p,L^q)$-boundedness of $I_{\gamma}$ (see Theorem \ref{strong}), we have
\begin{align}
I_1(y,r)&\leq w^q(B(y,r))^{1/{\beta}-1/q-1/s}\big\|I_\gamma(f_1)\big\|_{L^q(w^q)}\notag\\
&\leq C\cdot w^q(B(y,r))^{1/{\beta}-1/q-1/s}
\bigg(\int_{B(y,2r)}|f(x)|^p w^p(x)\,dx\bigg)^{1/p}\notag.
\end{align}
Observe that $1/{\beta}-1/q-1/s=1/{\alpha}-1/p-1/s$ when $1/{\beta}=1/{\alpha}-{\gamma}/n$. This fact implies that
\begin{align}\label{I1}
I_1(y,r)&\leq C\cdot w^q(B(y,r))^{1/{\alpha}-1/p-1/s}\big\|f\cdot\chi_{B(y,2r)}\big\|_{L^p(w^p)}\notag\\
&=C\cdot w^q(B(y,2r))^{1/{\alpha}-1/p-1/s}\big\|f\cdot\chi_{B(y,2r)}\big\|_{L^p(w^p)}\notag\\
&\times \frac{w^q(B(y,r))^{1/{\alpha}-1/p-1/s}}{w^q(B(y,2r))^{1/{\alpha}-1/p-1/s}}.
\end{align}
Since $w\in A_{p,q}$, we get $w^q\in A_q\subset A_{\infty}$ by Lemma \ref{relation} $(i)$. Moreover, since $1/{\alpha}-1/p-1/s<0$, then by doubling inequality (\ref{weights}), we obtain
\begin{equation}\label{doubling1}
\frac{w^q(B(y,r))^{1/{\alpha}-1/p-1/s}}{w^q(B(y,2r))^{1/{\alpha}-1/p-1/s}}\leq C.
\end{equation}
Substituting the above inequality \eqref{doubling1} into \eqref{I1}, we can see that
\begin{equation}\label{I1yr}
I_1(y,r)\leq C\cdot w^q(B(y,2r))^{1/{\alpha}-1/p-1/s}\big\|f\cdot\chi_{B(y,2r)}\big\|_{L^p(w^p)}.
\end{equation}
Let us now turn to the estimate of $I_2(y,r)$. First, it is clear that when $x\in B(y,r)$ and $z\in B(y,2r)^c$, we get $|x-z|\approx|y-z|$. We then decompose $\mathbb R^n$ into a geometrically increasing sequence of concentric balls, and obtain the following pointwise estimate:
\begin{align}\label{pointwise1}
\big|I_{\gamma}(f_2)(x)\big|
&\leq\frac{1}{\zeta(\gamma)}\int_{\mathbb R^n}\frac{|f_2(z)|}{|x-z|^{n-\gamma}}dz
\leq C\int_{B(y,2r)^c}\frac{|f(z)|}{|y-z|^{n-\gamma}}dz\notag\\
&=C\sum_{j=1}^\infty\int_{B(y,2^{j+1}r)\backslash B(y,2^jr)}\frac{|f(z)|}{|y-z|^{n-\gamma}}dz\notag\\
&\leq C\sum_{j=1}^\infty\frac{1}{|B(y,2^{j+1}r)|^{1-{\gamma}/n}}\int_{B(y,2^{j+1}r)}|f(z)|\,dz.
\end{align}
From this estimate \eqref{pointwise1}, it then follows that
\begin{equation*}
I_2(y,r)\leq C\cdot w^q(B(y,r))^{1/{\beta}-1/s}
\sum_{j=1}^\infty\frac{1}{|B(y,2^{j+1}r)|^{1-{\gamma}/n}}\int_{B(y,2^{j+1}r)}|f(z)|\,dz.
\end{equation*}
By using H\"older's inequality and $A_{p,q}$ condition on $w$, we get
\begin{equation*}
\begin{split}
&\frac{1}{|B(y,2^{j+1}r)|^{1-{\gamma}/n}}\int_{B(y,2^{j+1}r)}|f(z)|\,dz\\
&\leq\frac{1}{|B(y,2^{j+1}r)|^{1-{\gamma}/n}}\bigg(\int_{B(y,2^{j+1}r)}\big|f(z)\big|^pw^p(z)\,dz\bigg)^{1/p}
\bigg(\int_{B(y,2^{j+1}r)}w(z)^{-p'}dz\bigg)^{1/{p'}}\\
&\leq C\bigg(\int_{B(y,2^{j+1}r)}|f(z)|^pw^p(z)\,dz\bigg)^{1/p}\cdot w^q\big(B(y,2^{j+1}r)\big)^{-1/q}.
\end{split}
\end{equation*}
Hence,
\begin{equation}\label{I2yr}
\begin{split}
I_2(y,r)&\leq C\cdot w^q(B(y,r))^{1/{\beta}-1/s}\\
&\times\sum_{j=1}^\infty\bigg(\int_{B(y,2^{j+1}r)}|f(z)|^pw^p(z)\,dz\bigg)^{1/p}
\cdot w^q\big(B(y,2^{j+1}r)\big)^{-1/q}\\
&=C\sum_{j=1}^\infty w^q(B(y,2^{j+1}r))^{1/{\beta}-1/q-1/s}\big\|f\cdot\chi_{B(y,2^{j+1}r)}\big\|_{L^p(w^p)}\\
&\times\frac{w^q(B(y,r))^{1/{\beta}-1/s}}{w^q(B(y,2^{j+1}r))^{1/{\beta}-1/s}}\\
&=C\sum_{j=1}^\infty w^q(B(y,2^{j+1}r))^{1/{\alpha}-1/p-1/s}\big\|f\cdot\chi_{B(y,2^{j+1}r)}\big\|_{L^p(w^p)}\\
&\times\frac{w^q(B(y,r))^{1/{\beta}-1/s}}{w^q(B(y,2^{j+1}r))^{1/{\beta}-1/s}},
\end{split}
\end{equation}
where in the last equality we have used the relation $1/{\beta}-1/q=1/{\alpha}-1/p$. Notice that $w^q\in A_q\subset A_\infty$ for $1<q<\infty$, then by using the inequality (\ref{compare}) with exponent $\delta>0$ and our assumption $\beta<s$, we find that
\begin{align}\label{psi1}
\sum_{j=1}^\infty\frac{w^q(B(y,r))^{1/{\beta}-1/s}}{w^q(B(y,2^{j+1}r))^{1/{\beta}-1/s}}
&\leq C\sum_{j=1}^\infty\left(\frac{|B(y,r)|}{|B(y,2^{j+1}r)|}\right)^{\delta(1/{\beta}-1/s)}\notag\\
&=C\sum_{j=1}^\infty\left(\frac{1}{2^{(j+1)n}}\right)^{\delta(1/{\beta}-1/s)}\notag\\
&\leq C,
\end{align}
where the last series is convergent since $\delta(1/{\beta}-1/s)>0$. Therefore by taking the $L^s(\mu)$-norm of both sides of \eqref{I}(with respect to the variable $y$), and then using Minkowski's inequality, \eqref{I1yr}, \eqref{I2yr} and \eqref{psi1}, we have
\begin{equation*}
\begin{split}
&\Big\|w^q(B(y,r))^{1/{\beta}-1/q-1/s}\big\|I_\gamma(f)\cdot\chi_{B(y,r)}\big\|_{L^q(w^q)}\Big\|_{L^s(\mu)}\\
&\leq\big\|I_1(y,r)\big\|_{L^s(\mu)}+\big\|I_2(y,r)\big\|_{L^s(\mu)}\\
&\leq C\Big\|w^q(B(y,2r))^{1/{\alpha}-1/p-1/s}\big\|f\cdot\chi_{B(y,2r)}\big\|_{L^p(w^p)}\Big\|_{L^s(\mu)}\\
&+C\sum_{j=1}^\infty\Big\|w^q(B(y,2^{j+1}r))^{1/{\alpha}-1/p-1/s}\big\|f\cdot\chi_{B(y,2^{j+1}r)}\big\|_{L^p(w^p)}\Big\|_{L^s(\mu)}
\times\frac{w^q(B(y,r))^{1/{\beta}-1/s}}{w^q(B(y,2^{j+1}r))^{1/{\beta}-1/s}}\\
&\leq C\big\|f\big\|_{(L^p,L^s)^{\alpha}(w^p,w^q;\mu)}+C\big\|f\big\|_{(L^p,L^s)^{\alpha}(w^p,w^q;\mu)}
\times\sum_{j=1}^\infty\frac{w^q(B(y,r))^{1/{\beta}-1/s}}{w^q(B(y,2^{j+1}r))^{1/{\beta}-1/s}}\\
&\leq C\big\|f\big\|_{(L^p,L^s)^{\alpha}(w^p,w^q;\mu)}.
\end{split}
\end{equation*}
Thus, by taking the supremum over all $r>0$, we complete the proof of Theorem \ref{mainthm:1}.
\end{proof}

\begin{proof}[Proof of Theorem $\ref{mainthm:2}$]
Let $p=1$, $1\leq\alpha<s\leq\infty$ and $f\in(L^1,L^s)^{\alpha}(w,w^q;\mu)$ with $w\in A_{1,q}$ and $\mu\in\Delta_2$. For an arbitrary ball $B=B(y,r)$ in $\mathbb R^n$, we represent $f$ as
\begin{equation*}
f=f\cdot\chi_{2B}+f\cdot\chi_{(2B)^c}:=f_1+f_2;
\end{equation*}
then by the linearity of the fractional integral operator $I_{\gamma}$, one can write
\begin{align}\label{Iprime}
&w^q(B(y,r))^{1/{\beta}-1/q-1/s}\big\|I_{\gamma}(f)\cdot\chi_{B(y,r)}\big\|_{WL^q(w^q)}\notag\\
&\leq 2\cdot w^q(B(y,r))^{1/{\beta}-1/q-1/s}\big\|I_{\gamma}(f_1)\cdot\chi_{B(y,r)}\big\|_{WL^q(w^q)}\notag\\
&+2\cdot w^q(B(y,r))^{1/{\beta}-1/q-1/s}\big\|I_{\gamma}(f_2)\cdot\chi_{B(y,r)}\big\|_{WL^q(w^q)}\notag\\
&:=I'_1(y,r)+I'_2(y,r).
\end{align}
We first consider the term $I'_1(y,r)$. By the weighted weak $(1,q)$-boundedness of $I_{\gamma}$ (see Theorem \ref{weak}), we have
\begin{align}
I'_1(y,r)&\leq 2\cdot w^q(B(y,r))^{1/{\beta}-1/q-1/s}\big\|I_{\gamma}(f_1)\big\|_{WL^q(w^q)}\notag\\
&\leq C\cdot w^q(B(y,r))^{1/{\beta}-1/q-1/s}
\bigg(\int_{B(y,2r)}|f(x)| w(x)\,dx\bigg)\notag.
\end{align}
Observe that $1/{\beta}-1/q-1/s=1/{\alpha}-1-1/s$ when $1/{\beta}=1/{\alpha}-{\gamma}/n$ and $q=n/{(n-\gamma)}$. Then we have
\begin{align}\label{I1prime}
I'_1(y,r)&\leq C\cdot w^q(B(y,r))^{1/{\alpha}-1-1/s}
\bigg(\int_{B(y,2r)}|f(x)| w(x)\,dx\bigg)\notag\\
&=C\cdot w^q(B(y,2r))^{1/{\alpha}-1-1/s}\big\|f\cdot\chi_{B(y,2r)}\big\|_{L^1(w)}\notag\\
&\times \frac{w^q(B(y,r))^{1/{\alpha}-1-1/s}}{w^q(B(y,2r))^{1/{\alpha}-1-1/s}}.
\end{align}
Since $w$ is in the class $A_{1,q}$, we get $w^q\in A_1\subset A_{\infty}$ by Lemma \ref{relation} $(ii)$. Moreover, since $1/{\alpha}-1-1/s<0$, then we apply inequality (\ref{weights}) to obtain that
\begin{equation}\label{doubling2}
\frac{w^q(B(y,r))^{1/{\alpha}-1-1/s}}{w^q(B(y,2r))^{1/{\alpha}-1-1/s}}\leq C.
\end{equation}
Substituting the above inequality \eqref{doubling2} into \eqref{I1prime}, we thus obtain
\begin{equation}\label{WI1yr}
I'_1(y,r)\leq C\cdot w^q(B(y,2r))^{1/{\alpha}-1-1/s}\big\|f\cdot\chi_{B(y,2r)}\big\|_{L^1(w)}.
\end{equation}
As for the second term $I'_2(y,r)$, it follows directly from Chebyshev's inequality and the pointwise estimate \eqref{pointwise1} that
\begin{equation*}
\begin{split}
I'_2(y,r)&\leq2\cdot w^q(B(y,r))^{1/{\beta}-1/q-1/s}
\bigg(\int_{B(y,r)}\big|I_\gamma(f_2)(x)\big|^qw^q(x)\,dx\bigg)^{1/q}\\
&\leq C\cdot w^q(B(y,r))^{1/{\beta}-1/s}
\sum_{j=1}^\infty\frac{1}{|B(y,2^{j+1}r)|^{1-{\gamma}/n}}\int_{B(y,2^{j+1}r)}|f(z)|\,dz.
\end{split}
\end{equation*}
Moreover, by applying H\"older's inequality and then the reverse H\"older's inequality in succession, we can show that $w^q\in A_1$ if and only if $w\in A_1\cap RH_q$ (see \cite{johnson}), where $RH_q$ denotes the reverse H\"older class (see \cite{duoand} for further details). Another application of $A_1$ condition on $w$ gives that
\begin{equation*}
\begin{split}
&\frac{1}{|B(y,2^{j+1}r)|^{1-{\gamma}/n}}\int_{B(y,2^{j+1}r)}|f(z)|\,dz\\
&\leq C\cdot\frac{|B(y,2^{j+1}r)|^{\gamma/n}}{w(B(y,2^{j+1}r))}\cdot
\underset{z\in B(y,2^{j+1}r)}{\mbox{ess\,inf}}\;w(z)\int_{B(y,2^{j+1}r)}|f(z)|\,dz\\
&\leq C\cdot\frac{|B(y,2^{j+1}r)|^{\gamma/n}}{w(B(y,2^{j+1}r))}
\bigg(\int_{B(y,2^{j+1}r)}|f(z)|w(z)\,dz\bigg).
\end{split}
\end{equation*}
In addition, note that $w\in RH_q$. We are able to verify that for any positive integer $j\in\mathbb Z^+$,
\begin{equation*}
w^q\big(B(y,2^{j+1}r)\big)^{1/q}=\bigg(\int_{B(y,2^{j+1}r)}w^q(x)\,dx\bigg)^{1/q}\leq C\big|B(y,2^{j+1}r)\big|^{1/q-1}\cdot w\big(B(y,2^{j+1}r)\big),
\end{equation*}
which is equivalent to
\begin{equation}\label{wq}
\frac{|B(y,2^{j+1}r)|^{\gamma/n}}{w(B(y,2^{j+1}r))}\leq C\cdot\frac{1}{w^q(B(y,2^{j+1}r))^{1/q}}.
\end{equation}
Consequently,
\begin{equation}\label{WI2yr}
\begin{split}
I'_2(y,r)&\leq C\cdot w^q(B(y,r))^{1/{\beta}-1/s}\\
&\times\sum_{j=1}^\infty\bigg(\int_{B(y,2^{j+1}r)}|f(z)|w(z)\,dz\bigg)\cdot w^q\big(B(y,2^{j+1}r)\big)^{-1/q}\\
&=C\sum_{j=1}^\infty w^q\big(B(y,2^{j+1}r)\big)^{1/{\beta}-1/q-1/s}\big\|f\cdot\chi_{B(y,2^{j+1}r)}\big\|_{L^1(w)}\\
&\times\frac{w^q(B(y,r))^{1/{\beta}-1/s}}{w^q(B(y,2^{j+1}r))^{1/{\beta}-1/s}}\\
&=C\sum_{j=1}^\infty w^q\big(B(y,2^{j+1}r)\big)^{1/{\alpha}-1-1/s}\big\|f\cdot\chi_{B(y,2^{j+1}r)}\big\|_{L^1(w)}\\
&\times\frac{w^q(B(y,r))^{1/{\beta}-1/s}}{w^q(B(y,2^{j+1}r))^{1/{\beta}-1/s}},
\end{split}
\end{equation}
where in the last equality we have used the relation $1/{\beta}-1/q=1/{\alpha}-1$. Recall that $w^q\in A_1\subset A_\infty$, then by using the inequality (\ref{compare}) with exponent $\delta^\ast>0$ and the assumption $\beta<s$, we find that
\begin{align}\label{psi2}
\sum_{j=1}^\infty\frac{w^q(B(y,r))^{1/{\beta}-1/s}}{w^q(B(y,2^{j+1}r))^{1/{\beta}-1/s}}
&\leq C\sum_{j=1}^\infty\left(\frac{|B(y,r)|}{|B(y,2^{j+1}r)|}\right)^{\delta^\ast(1/{\beta}-1/s)}\notag\\
&= C\sum_{j=1}^\infty\left(\frac{1}{2^{(j+1)n}}\right)^{\delta^\ast(1/{\beta}-1/s)}\leq C,
\end{align}
where the last series is convergent since $\delta^\ast(1/{\beta}-1/s)>0$. Therefore by taking the $L^s(\mu)$-norm of both sides of \eqref{Iprime}(with respect to the variable $y$), and then using Minkowski's inequality, \eqref{WI1yr} and \eqref{WI2yr}, we have
\begin{equation*}
\begin{split}
&\Big\|w^q(B(y,r))^{1/{\beta}-1/q-1/s}\big\|I_{\gamma}(f)\cdot\chi_{B(y,r)}\big\|_{WL^q(w^q)}\Big\|_{L^s(\mu)}\\
&\leq\big\|I'_1(y,r)\big\|_{L^s(\mu)}+\big\|I'_2(y,r)\big\|_{L^s(\mu)}\\
&\leq C\Big\|w^q(B(y,2r))^{1/{\alpha}-1-1/s}\big\|f\cdot\chi_{B(y,2r)}\big\|_{L^1(w)}\Big\|_{L^s(\mu)}\\
&+C\sum_{j=1}^\infty\Big\|w^q(B(y,2^{j+1}r))^{1/{\alpha}-1-1/s}\big\|f\cdot\chi_{B(y,2^{j+1}r)}\big\|_{L^1(w)}\Big\|_{L^s(\mu)}
\times\frac{w^q(B(y,r))^{1/{\beta}-1/s}}{w^q(B(y,2^{j+1}r))^{1/{\beta}-1/s}}\\
&\leq C\big\|f\big\|_{(L^1,L^s)^{\alpha}(w,w^q;\mu)}+C\big\|f\big\|_{(L^1,L^s)^{\alpha}(w,w^q;\mu)}
\times\sum_{j=1}^\infty\frac{w^q(B(y,r))^{1/{\beta}-1/s}}{w^q(B(y,2^{j+1}r))^{1/{\beta}-1/s}}\\
&\leq C\big\|f\big\|_{(L^1,L^s)^{\alpha}(w,w^q;\mu)},
\end{split}
\end{equation*}
where the last inequality follows from \eqref{psi2}. Thus, by taking the supremum over all $r>0$, we finish the proof of Theorem \ref{mainthm:2}.
\end{proof}

Given $0<\gamma<n$, the related fractional maximal operator $M_{\gamma}$ with order $\gamma$ is given by
\begin{equation*}
M_{\gamma}f(x):=\sup_{B\ni x}\frac{1}{|B|^{1-\gamma/n}}\int_B|f(y)|\,dy,
\end{equation*}
where the supremum is taken over all balls $B$ containing $x$. Let us point out that $M_{\gamma}f(x)$ can be controlled pointwise by $I_\gamma(|f|)(x)$ for any $f(x)$. In fact, fix $r>0$, then we have
\begin{equation*}
\begin{split}
I_\gamma(|f|)(x)&\ge\int_{|y-x|<r}\frac{|f(y)|}{|x-y|^{n-\gamma}}\,dy\\
&\geq\frac{1}{r^{n-\gamma}}\int_{|y-x|<r}|f(y)|\,dy.
\end{split}
\end{equation*}
Taking the supremum for $r>0$ on both sides of the above inequality, we get
\begin{equation}\label{dominate2}
I_\gamma(|f|)(x)\geq\sup_{r>0}\frac{1}{r^{n-\gamma}}\int_{|y-x|<r}|f(y)|\,dy=C\cdot M_{\gamma}f(x).
\end{equation}
This is just our desired conclusion. An immediate application of the above inequality \eqref{dominate2} is the following weighted strong-type and weak-type estimates for the operator $M_{\gamma}$.
\begin{corollary}
Let $0<\gamma<n$, $1<p<n/{\gamma}$, $1/q=1/p-{\gamma}/n$ and $w\in A_{p,q}$. Assume that $p\leq\alpha<\beta<s\leq\infty$ and $\mu\in\Delta_2$, then the fractional maximal operator $M_{\gamma}$ is bounded from $(L^p,L^s)^{\alpha}(w^p,w^q;\mu)$ into $(L^q,L^s)^{\beta}(w^q;\mu)$ with $1/{\beta}=1/{\alpha}-{\gamma}/n$.
\end{corollary}

\begin{corollary}
Let $0<\gamma<n$, $p=1$, $q=n/{(n-\gamma)}$ and $w\in A_{1,q}$. Assume that $1\leq\alpha<\beta<s\leq\infty$ and $\mu\in\Delta_2$, then the fractional maximal operator $M_{\gamma}$ is bounded from $(L^1,L^s)^{\alpha}(w,w^q;\mu)$ into $(WL^q,L^s)^{\beta}(w^q;\mu)$ with $1/{\beta}=1/{\alpha}-{\gamma}/n$.
\end{corollary}

Suppose that $\mathcal L$ is a linear operator which generates an analytic semigroup $\big\{e^{-t\mathcal L}\big\}_{t>0}$ on $L^2(\mathbb R^n)$ with a kernel $p_t(x,y)$ satisfying Gaussian upper bound; that is, there exist two positive constants $C$ and $A$ such that for all $x,y\in\mathbb R^n$ and all $t>0$, we have
\begin{equation}\label{G}
\big|p_t(x,y)\big|\le\frac{C}{t^{n/2}}e^{-A\frac{|x-y|^2}{t}}.
\end{equation}
For any $0<\gamma<n$, the generalized fractional integrals $\mathcal L^{-\gamma/2}$ associated to the operator $\mathcal L$ is defined by
\begin{equation}\label{gefrac}
\mathcal L^{-\gamma/2}f(x):=\frac{1}{\Gamma(\gamma/2)}\int_0^\infty e^{-t\mathcal L}(f)(x)t^{\gamma/2-1}\,dt.
\end{equation}
Note that if $\mathcal L=-\Delta$ is the Laplacian on $\mathbb R^n$, then $\mathcal L^{-\gamma/2}$ is the classical fractional integral operator $I_\gamma$, which is given by \eqref{frac}. Since the semigroup $e^{-t\mathcal L}$ has a kernel $p_t(x,y)$ which satisfies the Gaussian upper bound \eqref{G}, it is easy to check that for all $x\in\mathbb R^n$,
\begin{equation}\label{dominate1}
\big|\mathcal L^{-\gamma/2}(f)(x)\big|\le C\cdot I_\gamma(|f|)(x).
\end{equation}
In fact, if we denote the kernel of $\mathcal L^{-\gamma/2}$ by $\mathcal K_\gamma(x,y)$, then it follows immediately from \eqref{gefrac} that (see \cite{duong1,mo})
\begin{equation}\label{kgamma}
\mathcal K_\gamma(x,y)=\frac{1}{\Gamma(\gamma/2)}\int_0^\infty p_t(x,y)t^{\gamma/2-1}\,dt,
\end{equation}
where $p_t(x,y)$ is the kernel of $e^{-t\mathcal L}$. Thus, by using the Gaussian upper bound (\ref{G}) and the expression (\ref{kgamma}), we can deduce that (see \cite{duong1} and \cite{mo})
\begin{align}\label{kernelk}
\big|\mathcal K_\gamma(x,y)\big|
&\leq\frac{1}{\Gamma(\gamma/2)}\int_0^\infty\big|p_t(x,y)\big|t^{\gamma/2-1}\,dt\notag\\
&\leq C\int_0^\infty e^{-A\frac{|x-y|^2}{t}}t^{\gamma/2-n/2-1}\,dt\notag\\
&\leq C\cdot\frac{1}{|x-y|^{n-\gamma}},
\end{align}
which implies \eqref{dominate1}. Taking into account this pointwise inequality, as a consequence of Theorems \ref{mainthm:1} and \ref{mainthm:2}, we have the following results.
\begin{corollary}
Let $0<\gamma<n$, $1<p<n/{\gamma}$, $1/q=1/p-{\gamma}/n$ and $w\in A_{p,q}$. Assume that $p\leq\alpha<\beta<s\leq\infty$ and $\mu\in\Delta_2$, then the generalized fractional integrals $\mathcal L^{-\gamma/2}$ is bounded from $(L^p,L^s)^{\alpha}(w^p,w^q;\mu)$ into $(L^q,L^s)^{\beta}(w^q;\mu)$ with $1/{\beta}=1/{\alpha}-{\gamma}/n$.
\end{corollary}

\begin{corollary}
Let $0<\gamma<n$, $p=1$, $q=n/{(n-\gamma)}$ and $w\in A_{1,q}$. Assume that $1\leq\alpha<\beta<s\leq\infty$ and $\mu\in\Delta_2$, then the generalized fractional integrals $\mathcal L^{-\gamma/2}$ is bounded from $(L^1,L^s)^{\alpha}(w,w^q;\mu)$ into $(WL^q,L^s)^{\beta}(w^q;\mu)$ with $1/{\beta}=1/{\alpha}-{\gamma}/n$.
\end{corollary}

\section{Proofs of Theorems \ref{mainthm:3} and \ref{mainthm:4}}

To prove our main theorems in this section, we need the following lemma about $BMO(\mathbb R^n)$ functions.

\begin{lemma}\label{BMO}
Let $b$ be a function in $BMO(\mathbb R^n)$.

$(i)$ For any ball $B$ in $\mathbb R^n$ and for any positive integer $j\in\mathbb Z^+$, then
\begin{equation*}
\big|b_{2^{j+1}B}-b_B\big|\leq C\cdot(j+1)\|b\|_*.
\end{equation*}

$(ii)$ Let $1<q<\infty$. For any ball $B$ in $\mathbb R^n$ and for any weight $\nu\in A_{\infty}$, then
\begin{equation*}
\bigg(\int_B\big|b(x)-b_B\big|^q\nu(x)\,dx\bigg)^{1/q}\leq C\|b\|_*\cdot\nu(B)^{1/q}.
\end{equation*}
\end{lemma}
\begin{proof}
For the proof of $(i)$, we refer the reader to \cite{stein2}. For the proof of $(ii)$, we refer the reader to \cite{wang}.
\end{proof}

\begin{proof}[Proof of Theorem $\ref{mainthm:3}$]
Let $1<p\leq\alpha<s\leq\infty$ and $f\in(L^p,L^s)^{\alpha}(w^p,w^q;\mu)$ with $w\in A_{p,q}$ and $\mu\in\Delta_2$.
For each fixed ball $B=B(y,r)$ in $\mathbb R^n$, as before, we represent $f$ as $f=f_1+f_2$, where $f_1=f\cdot\chi_{2B}$, $2B=B(y,2r)\subset\mathbb R^n$. By the linearity of the commutator operator $[b,I_{\gamma}]$, we write
\begin{align}\label{J}
&w^q(B(y,r))^{1/{\beta}-1/q-1/s}\big\|[b,I_\gamma](f)\cdot\chi_{B(y,r)}\big\|_{L^q(w^q)}\notag\\
&=w^q(B(y,r))^{1/{\beta}-1/q-1/s}\bigg(\int_{B(y,r)}\big|[b,I_{\gamma}](f)(x)\big|^qw^q(x)\,dx\bigg)^{1/q}\notag\\
&\leq w^q(B(y,r))^{1/{\beta}-1/q-1/s}\bigg(\int_{B(y,r)}\big|[b,I_\gamma](f_1)(x)\big|^qw^q(x)\,dx\bigg)^{1/q}\notag\\
&+w^q(B(y,r))^{1/{\beta}-1/q-1/s}\bigg(\int_{B(y,r)}\big|[b,I_\gamma](f_2)(x)\big|^qw^q(x)\,dx\bigg)^{1/q}\notag\\
&:=J_1(y,r)+J_2(y,r).
\end{align}
Since $w$ is in the class $A_{p,q}$, we get $w^q\in A_q\subset A_{\infty}$ by Lemma \ref{relation}$(i)$. Also observe that $1/{\beta}-1/q=1/{\alpha}-1/p$ by our assumption. By using Theorem \ref{cstrong}, we obtain
\begin{align}\label{J1yr}
J_1(y,r)&\leq w^q(B(y,r))^{1/{\beta}-1/q-1/s}\big\|[b,I_{\gamma}](f_1)\big\|_{L^q(w^q)}\notag\\
&\leq C\cdot  w^q(B(y,r))^{1/{\beta}-1/q-1/s}
\bigg(\int_{B(y,2r)}|f(x)|^p w^p(x)\,dx\bigg)^{1/p}\notag\\
&=C\cdot w^q(B(y,2r))^{1/{\alpha}-1/p-1/s}\big\|f\cdot\chi_{B(y,2r)}\big\|_{L^p(w^p)}\notag\\
&\times \frac{w^q(B(y,r))^{1/{\alpha}-1/p-1/s}}{w^q(B(y,2r))^{1/{\alpha}-1/p-1/s}}\notag\\
&\leq C\cdot w^q(B(y,2r))^{1/{\alpha}-1/p-1/s}\big\|f\cdot\chi_{B(y,2r)}\big\|_{L^p(w^p)},
\end{align}
where the last inequality is due to (\ref{weights}) and the fact that $1/{\alpha}-1/p-1/s<0$. Let us now turn to the estimate of $J_2(y,r)$. By definition, for any $x\in B(y,r)$, we have
\begin{equation*}
\big|[b,I_{\gamma}](f_2)(x)\big|\leq\big|b(x)-b_{B(y,r)}\big|\cdot\big|I_{\gamma}(f_2)(x)\big|
+\Big|I_{\gamma}\big([b_{B(y,r)}-b]f_2\big)(x)\Big|.
\end{equation*}
In the proof of Theorem \ref{mainthm:1}, we have already shown that (see \eqref{pointwise1})
\begin{equation}\label{wh8}
\big|I_{\gamma}(f_2)(x)\big|\leq C\sum_{j=1}^\infty\frac{1}{|B(y,2^{j+1}r)|^{1-{\gamma}/n}}\int_{B(y,2^{j+1}r)}|f(z)|\,dz.
\end{equation}
By the same manner as in the proof of \eqref{pointwise1}, we can also show that
\begin{align}\label{pointwise2}
\Big|I_{\gamma}\big([b_{B(y,r)}-b]f_2\big)(x)\Big|
&\leq\frac{1}{\zeta(\gamma)}\int_{\mathbb R^n}\frac{|[b_{B(y,r)}-b(z)]f_2(z)|}{|x-z|^{n-\gamma}}dz\\
&\leq C\int_{B(y,2r)^c}\frac{|[b_{B(y,r)}-b(z)]f(z)|}{|y-z|^{n-\gamma}}dz\notag\\
&=C\sum_{j=1}^\infty\int_{B(y,2^{j+1}r)\backslash B(y,2^jr)}\frac{|b(z)-b_{B(y,r)}|\cdot|f(z)|}{|y-z|^{n-\gamma}}dz\notag\\
&\leq C\sum_{j=1}^\infty
\frac{1}{|B(y,2^{j+1}r)|^{1-{\gamma}/n}}\int_{B(y,2^{j+1}r)}\big|b(z)-b_{B(y,r)}\big|\cdot\big|f(z)\big|\,dz\notag.
\end{align}
Hence, from the above two pointwise estimates \eqref{wh8} and \eqref{pointwise2}, it follows that
\begin{equation*}
\begin{split}
J_2(y,r)&\leq C\cdot w^q(B(y,r))^{1/{\beta}-1/q-1/s}\bigg(\int_{B(y,r)}\big|b(x)-b_{B(y,r)}\big|^qw^q(x)\,dx\bigg)^{1/q}\\
&\times\sum_{j=1}^\infty\bigg(\frac{1}{|B(y,2^{j+1}r)|^{1-{\gamma}/n}}\int_{B(y,2^{j+1}r)}|f(z)|\,dz\bigg)\\
&+C\cdot w^q(B(y,r))^{1/{\beta}-1/s}
\sum_{j=1}^\infty\frac{1}{|B(y,2^{j+1}r)|^{1-{\gamma}/n}}\int_{B(y,2^{j+1}r)}\big|b_{B(y,2^{j+1}r)}-b_{B(y,r)}\big|\cdot|f(z)|\,dz\\
&+C\cdot w^q(B(y,r))^{1/{\beta}-1/s}
\sum_{j=1}^\infty\frac{1}{|B(y,2^{j+1}r)|^{1-{\gamma}/n}}\int_{B(y,2^{j+1}r)}\big|b(z)-b_{B(y,2^{j+1}r)}\big|\cdot|f(z)|\,dz\\
&:=J_3(y,r)+J_4(y,r)+J_5(y,r).
\end{split}
\end{equation*}
Below we will give the estimates of $J_3(y,r)$, $J_4(y,r)$ and $J_5(y,r)$, respectively.
To estimate $J_3(y,r)$, note that $w^q\in A_q\subset A_{\infty}$ with $1<q<\infty$. Using the second part of Lemma \ref{BMO}, H\"older's inequality and the $A_{p,q}$ condition on $w$, we compute
\begin{equation*}
\begin{split}
J_3(y,r)&\leq C\|b\|_*\cdot w^q(B(y,r))^{1/{\beta}-1/s}
\times\sum_{j=1}^\infty\bigg(\frac{1}{|B(y,2^{j+1}r)|^{1-{\gamma}/n}}\int_{B(y,2^{j+1}r)}|f(z)|\,dz\bigg)\\
&\leq C\|b\|_*\cdot w^q(B(y,r))^{1/{\beta}-1/s}
\sum_{j=1}^\infty\frac{1}{|B(y,2^{j+1}r)|^{1-{\gamma}/n}}
\bigg(\int_{B(y,2^{j+1}r)}|f(z)|^pw^p(z)\,dz\bigg)^{1/p}\\
&\times\bigg(\int_{B(y,2^{j+1}r)}w(z)^{-{p'}}\,dz\bigg)^{1/{p'}}\\
&\leq C\|b\|_*\cdot w^q(B(y,r))^{1/{\beta}-1/s}\\
&\times\sum_{j=1}^\infty\bigg(\int_{B(y,2^{j+1}r)}|f(z)|^pw^p(z)\,dz\bigg)^{1/p}
\cdot w^q\big(B(y,2^{j+1}r)\big)^{-1/q}.
\end{split}
\end{equation*}
To estimate $J_4(y,r)$, applying the first part of Lemma \ref{BMO}, H\"older's inequality and the $A_{p,q}$ condition on $w$, we can deduce that
\begin{equation*}
\begin{split}
J_4(y,r)&\leq C\|b\|_*\cdot w^q(B(y,r))^{1/{\beta}-1/s}
\times\sum_{j=1}^\infty\frac{(j+1)}{|B(y,2^{j+1}r)|^{1-{\gamma}/n}}\int_{B(y,2^{j+1}r)}|f(z)|\,dz\\
&\leq C\|b\|_*\cdot w^q(B(y,r))^{1/{\beta}-1/s}
\sum_{j=1}^\infty\frac{(j+1)}{|B(y,2^{j+1}r)|^{1-{\gamma}/n}}
\bigg(\int_{B(y,2^{j+1}r)}|f(z)|^pw^p(z)\,dz\bigg)^{1/p}\\
&\times\bigg(\int_{B(y,2^{j+1}r)}w(z)^{-{p'}}\,dz\bigg)^{1/{p'}}\\
&\leq C\|b\|_*\cdot w^q(B(y,r))^{1/{\beta}-1/s}\\
&\times\sum_{j=1}^\infty\big(j+1\big)\cdot\bigg(\int_{B(y,2^{j+1}r)}|f(z)|^pw^p(z)\,dz\bigg)^{1/p}
\cdot w^q\big(B(y,2^{j+1}r)\big)^{-1/q}.
\end{split}
\end{equation*}
It remains to estimate the last term $J_5(y,r)$. An application of H\"older's inequality gives us that
\begin{equation*}
\begin{split}
J_5(y,r)&\leq C\cdot w^q(B(y,r))^{1/{\beta}-1/s}\sum_{j=1}^\infty\frac{1}{|B(y,2^{j+1}r)|^{1-{\gamma}/n}}
\bigg(\int_{B(y,2^{j+1}r)}|f(z)|^pw^p(z)\,dz\bigg)^{1/p}\\
&\times\bigg(\int_{B(y,2^{j+1}r)}\big|b(z)-b_{B(y,2^{j+1}r)}\big|^{p'}w(z)^{-{p'}}\,dz\bigg)^{1/{p'}}.
\end{split}
\end{equation*}
If we set $\nu(z)=w(z)^{-{p'}}$, then we have $\nu\in A_{p'}\subset A_{\infty}$ because $w\in A_{p,q}$ by Lemma \ref{relation}$(i)$. Thus, it follows from the second part of Lemma \ref{BMO} and the $A_{p,q}$ condition on $w$ that
\begin{align}\label{WBMO}
\bigg(\int_{B(y,2^{j+1}r)}\big|b(z)-b_{B(y,2^{j+1}r)}\big|^{p'}\nu(z)\,dz\bigg)^{1/{p'}}
&\leq C\|b\|_*\cdot \nu\big(B(y,2^{j+1}r)\big)^{1/{p'}}\notag\\
&=C\|b\|_*\cdot\bigg(\int_{B(y,2^{j+1}r)}w(z)^{-{p'}}dz\bigg)^{1/{p'}}\notag\\
&\leq C\|b\|_*\cdot\frac{|B(y,2^{j+1}r)|^{1-{\gamma}/n}}{w^q(B(y,2^{j+1}r))^{1/q}}.
\end{align}
Therefore, in view of the estimate \eqref{WBMO}, we get
\begin{equation*}
\begin{split}
J_5(y,r)&\leq C\|b\|_*\cdot w^q(B(y,r))^{1/{\beta}-1/s}\\
&\times\sum_{j=1}^\infty\bigg(\int_{B(y,2^{j+1}r)}|f(z)|^pw^p(z)\,dz\bigg)^{1/p}
\cdot w^q\big(B(y,2^{j+1}r)\big)^{-1/q}.
\end{split}
\end{equation*}
Summarizing the estimates derived above, we conclude that
\begin{align}\label{J2yr}
J_2(y,r)&\leq C\|b\|_*\cdot w^q(B(y,r))^{1/{\beta}-1/s}\notag\\
&\times\sum_{j=1}^\infty\big(j+1\big)\cdot\bigg(\int_{B(y,2^{j+1}r)}|f(z)|^pw^p(z)\,dz\bigg)^{1/p}
\cdot w^q\big(B(y,2^{j+1}r)\big)^{-1/q}\notag\\
&=C\|b\|_*\sum_{j=1}^\infty w^q\big(B(y,2^{j+1}r)\big)^{1/{\beta}-1/q-1/s}\big\|f\cdot\chi_{B(y,2^{j+1}r)}\big\|_{L^p(w^p)}\notag\\
&\times\big(j+1\big)\cdot\frac{w^q(B(y,r))^{1/{\beta}-1/s}}{w^q(B(y,2^{j+1}r))^{1/{\beta}-1/s}}\notag\\
&=C\|b\|_*\sum_{j=1}^\infty w^q\big(B(y,2^{j+1}r)\big)^{1/{\alpha}-1/p-1/s}\big\|f\cdot\chi_{B(y,2^{j+1}r)}\big\|_{L^p(w^p)}\notag\\
&\times\big(j+1\big)\cdot\frac{w^q(B(y,r))^{1/{\beta}-1/s}}{w^q(B(y,2^{j+1}r))^{1/{\beta}-1/s}},
\end{align}
where in the last equality we have used the relation $1/{\beta}-1/q=1/{\alpha}-1/p$ again. Since $w^q\in A_q$ with $1<q<\infty$, then by using the inequality (\ref{compare}) with exponent $\delta>0$  together with the fact that $\beta<s$, we obtain
\begin{align}\label{psi3}
\sum_{j=1}^\infty\big(j+1\big)\cdot\frac{w^q(B(y,r))^{1/{\beta}-1/s}}{w^q(B(y,2^{j+1}r))^{1/{\beta}-1/s}}
&\leq C\sum_{j=1}^\infty\big(j+1\big)\cdot\left(\frac{|B(y,r)|}{|B(y,2^{j+1}r)|}\right)^{\delta(1/{\beta}-1/s)}\notag\\
&= C\sum_{j=1}^\infty\big(j+1\big)\cdot\left(\frac{1}{2^{(j+1)n}}\right)^{\delta(1/{\beta}-1/s)}\notag\\
&\leq C,
\end{align}
where the last series is convergent since the exponent $\delta(1/{\beta}-1/s)$ is positive.
Therefore by taking the $L^s(\mu)$-norm of both sides of \eqref{J}(with respect to the variable $y$), and then using Minkowski's inequality, \eqref{J1yr} and \eqref{J2yr}, we can get
\begin{equation*}
\begin{split}
&\Big\|w^q(B(y,r))^{1/{\beta}-1/q-1/s}\big\|[b,I_\gamma](f)\cdot\chi_{B(y,r)}\big\|_{L^q(w^q)}\Big\|_{L^s(\mu)}\\
&\leq\big\|J_1(y,r)\big\|_{L^s(\mu)}+\big\|J_2(y,r)\big\|_{L^s(\mu)}\\
&\leq C\Big\|w^q(B(y,2r))^{1/{\alpha}-1/p-1/s}\big\|f\cdot\chi_{B(y,2r)}\big\|_{L^p(w^p)}\Big\|_{L^s(\mu)}\\
&+C\sum_{j=1}^\infty\Big\|w^q\big(B(y,2^{j+1}r)\big)^{1/{\alpha}-1/p-1/s}\big\|f\cdot\chi_{B(y,2^{j+1}r)}\big\|_{L^p(w^p)}\Big\|_{L^s(\mu)}\\
&\times\big(j+1\big)\cdot\frac{w^q(B(y,r))^{1/{\beta}-1/s}}{w^q(B(y,2^{j+1}r))^{1/{\beta}-1/s}}\\
&\leq C\big\|f\big\|_{(L^p,L^s)^{\alpha}(w^p,w^q;\mu)}+C\big\|f\big\|_{(L^p,L^s)^{\alpha}(w^p,w^q;\mu)}
\times\sum_{j=1}^\infty\big(j+1\big)\cdot\frac{w^q(B(y,r))^{1/{\beta}-1/s}}{w^q(B(y,2^{j+1}r))^{1/{\beta}-1/s}}\\
&\leq C\big\|f\big\|_{(L^p,L^s)^{\alpha}(w^p,w^q;\mu)},
\end{split}
\end{equation*}
where the last inequality follows from \eqref{psi3}. Thus, by taking the supremum over all $r>0$, we complete the proof of Theorem \ref{mainthm:3}.
\end{proof}

\begin{proof}[Proof of Theorem $\ref{mainthm:4}$]
For any fixed ball $B=B(y,r)$ in $\mathbb R^n$, as before, we represent $f$ as $f=f_1+f_2$, where $f_1=f\cdot\chi_{2B}$ and $f_2=f\cdot\chi_{(2B)^c}$. Then for any given $\lambda>0$, by the linearity of the commutator operator $[b,I_{\gamma}]$, one can write
\begin{align}\label{Jprime}
&w^q(B(y,r))^{1/{\beta}-1/q-1/s}\cdot\Big[w^q\Big(\Big\{x\in B(y,r):\big|[b,I_{\gamma}](f)(x)\big|>\lambda\Big\}\Big)\Big]^{1/q}\notag\\
\leq &w^q(B(y,r))^{1/{\beta}-1/q-1/s}\cdot \Big[w^q\Big(\Big\{x\in B(y,r):\big|[b,I_{\gamma}](f_1)(x)\big|>\lambda/2\Big\}\Big)\Big]^{1/q}\notag\\
&+w^q(B(y,r))^{1/{\beta}-1/q-1/s}\cdot \Big[w^q\Big(\Big\{x\in B(y,r):\big|[b,I_{\gamma}](f_2)(x)\big|>\lambda/2\Big\}\Big)\Big]^{1/q}\notag\\
:=&J'_1(y,r)+J'_2(y,r).
\end{align}
We first consider the term $J'_1(y,r)$. By using Theorem \ref{cweak}, we get
\begin{equation*}
\begin{split}
J'_1(y,r)
&\leq C\cdot w^q(B(y,r))^{1/{\beta}-1/q-1/s}\int_{B(y,2r)}\Phi\bigg(\frac{|f(x)|}{\lambda}\bigg)\cdot w(x)\,dx\\
&=C\cdot w^q(B(y,r))^{1/{\alpha}-1-1/s}\int_{B(y,2r)}\Phi\bigg(\frac{|f(x)|}{\lambda}\bigg)\cdot w(x)\,dx,
\end{split}
\end{equation*}
where in the last equality we have used our assumption $1/{\beta}=1/{\alpha}-{\gamma}/n$.
Since $w$ is a weight in the class $A_{1,q}$, one has $w^q\in A_1\subset A_{\infty}$ by Lemma \ref{relation}$(ii)$. This fact, together with the inequalities \eqref{doubling2} and \eqref{main esti1}, gives us that
\begin{align}\label{WJ1yr}
J'_1(y,r)&\leq C\cdot\frac{w^q(B(y,2r))^{1/{\alpha}-1-1/s}w(B(y,2r))}{w(B(y,2r))}\int_{B(y,2r)}\Phi\left(\frac{|f(x)|}{\lambda}\right)\cdot w(x)\,dx\notag\\
&\leq C\cdot w^q(B(y,2r))^{1/{\alpha}-1-1/s}w(B(y,2r))\bigg\|\Phi\left(\frac{|f|}{\,\lambda\,}\right)\bigg\|_{L\log L(w),B(y,2r)}.
\end{align}
We now turn to deal with the term $J'_2(y,r)$. Recall that the following inequality
\begin{equation*}
\big|[b,I_{\gamma}](f_2)(x)\big|\leq\big|b(x)-b_{B(y,r)}\big|\cdot\big|I_{\gamma}(f_2)(x)\big|
+\Big|I_{\gamma}\big([b_{B(y,r)}-b]f_2\big)(x)\Big|
\end{equation*}
is valid. Thus, we can further decompose $J'_2(y,r)$ as
\begin{equation*}
\begin{split}
J'_2(y,r)\leq&w^q(B(y,r))^{1/{\beta}-1/q-1/s}\cdot
\Big[w^q\Big(\Big\{x\in B(y,r):\big|b(x)-b_{B(y,r)}\big|\cdot\big|I_{\gamma}(f_2)(x)\big|>\lambda/4\Big\}\Big)\Big]^{1/q}\\
&+w^q(B(y,r))^{1/{\beta}-1/q-1/s}\cdot
\Big[w^q\Big(\Big\{x\in B(y,r):\Big|I_{\gamma}\big([b_{B(y,r)}-b]f_2\big)(x)\Big|>\lambda/4\Big\}\Big)\Big]^{1/q}\\
:=&J'_3(y,r)+J'_4(y,r).
\end{split}
\end{equation*}
Applying the previous pointwise estimate \eqref{pointwise1}, Chebyshev's inequality together with Lemma \ref{BMO}$(ii)$, we deduce that
\begin{equation*}
\begin{split}
J'_3(y,r)&\leq w^q(B(y,r))^{1/{\beta}-1/q-1/s}\cdot\frac{\,4\,}{\lambda}
\bigg(\int_{B(y,r)}\big|b(x)-b_{B(y,r)}\big|^q\cdot\big|I_{\gamma}(f_2)(x)\big|^qw^q(x)\,dx\bigg)^{1/q}\\
&\leq C\cdot w^q(B(y,r))^{1/{\beta}-1/s}
\sum_{j=1}^\infty\frac{1}{|B(y,2^{j+1}r)|^{1-{\gamma}/n}}\int_{B(y,2^{j+1}r)}\frac{|f(z)|}{\lambda}\,dz\\
&\times\bigg(\frac{1}{w^q(B(y,r))}\int_{B(y,r)}\big|b(x)-b_{B(y,r)}\big|^qw^q(x)\,dx\bigg)^{1/q}\\
&\leq C\|b\|_*\sum_{j=1}^\infty\frac{1}{|B(y,2^{j+1}r)|^{1-{\gamma}/n}}
\int_{B(y,2^{j+1}r)}\frac{|f(z)|}{\lambda}\,dz\times w^q(B(y,r))^{1/{\beta}-1/s}.
\end{split}
\end{equation*}
Furthermore, note that $t\leq\Phi(t)=t\cdot(1+\log^+t)$ for any $t>0$. As we pointed out in Theorem \ref{mainthm:2} that $w^q\in A_1$ if and only if $w\in A_1\cap RH_q$, it then follows from the $A_1$ condition that
\begin{equation*}
\begin{split}
J'_3(y,r)&\leq C\|b\|_*\sum_{j=1}^\infty\frac{|B(y,2^{j+1}r)|^{{\gamma}/n}}{w(B(y,2^{j+1}r))}
\int_{B(y,2^{j+1}r)}\frac{|f(z)|}{\lambda}\cdot w(z)\,dz
\times w^q(B(y,r))^{1/{\beta}-1/s}\\
&\leq C\|b\|_*\sum_{j=1}^\infty\frac{|B(y,2^{j+1}r)|^{{\gamma}/n}}{w(B(y,2^{j+1}r))}
\int_{B(y,2^{j+1}r)}\Phi\left(\frac{|f(z)|}{\lambda}\right)\cdot w(z)\,dz
\times w^q(B(y,r))^{1/{\beta}-1/s}\\
&\leq C\|b\|_*\sum_{j=1}^\infty\bigg\|\Phi\left(\frac{|f|}{\,\lambda\,}\right)\bigg\|_{L\log L(w),B(y,2^{j+1}r)}
\times\big|B(y,2^{j+1}r)\big|^{{\gamma}/n}w^q(B(y,r))^{1/{\beta}-1/s},
\end{split}
\end{equation*}
where in the last inequality we have used the estimate \eqref{main esti1}. In view of \eqref{wq} and our assumption $1/{\beta}=1/{\alpha}-{\gamma}/n$, we have
\begin{equation*}
\begin{split}
J'_3(y,r)&\leq C\|b\|_*\sum_{j=1}^\infty\bigg\|\Phi\left(\frac{|f|}{\,\lambda\,}\right)\bigg\|_{L\log L(w),B(y,2^{j+1}r)}\times \frac{w^q(B(y,r))^{1/{\beta}-1/s}}{w^q(B(y,2^{j+1}r))^{1/q}}w\big(B(y,2^{j+1}r)\big)\\
&=C\|b\|_*\sum_{j=1}^\infty\bigg\|\Phi\left(\frac{|f|}{\,\lambda\,}\right)\bigg\|_{L\log L(w),B(y,2^{j+1}r)}\\
&\times w^q\big(B(y,2^{j+1}r)\big)^{1/{\alpha}-1-1/s}w\big(B(y,2^{j+1}r)\big)
\cdot\frac{w^q(B(y,r))^{1/{\beta}-1/s}}{w^q(B(y,2^{j+1}r))^{1/{\beta}-1/s}}.
\end{split}
\end{equation*}
On the other hand, applying the pointwise estimate \eqref{pointwise2} and Chebyshev's inequality, we get
\begin{equation*}
\begin{split}
J'_4(y,r)&\leq w^q(B(y,r))^{1/{\beta}-1/q-1/s}\cdot\frac{\,4\,}{\lambda}
\bigg(\int_{B(y,r)}\Big|I_{\gamma}\big([b_{B(y,r)}-b]f_2\big)(x)\Big|^qw^q(x)\,dx\bigg)^{1/q}\\
&\leq w^q(B(y,r))^{1/{\beta}-1/s}\cdot\frac{\,C\,}{\lambda}
\sum_{j=1}^\infty\frac{1}{|B(y,2^{j+1}r)|^{1-{\gamma}/n}}\int_{B(y,2^{j+1}r)}\big|b(z)-b_{B(y,r)}\big|\cdot|f(z)|\,dz\\
&\leq w^q(B(y,r))^{1/{\beta}-1/s}\cdot\frac{\,C\,}{\lambda}
\sum_{j=1}^\infty\frac{1}{|B(y,2^{j+1}r)|^{1-{\gamma}/n}}\int_{B(y,2^{j+1}r)}\big|b(z)-b_{B(y,2^{j+1}r)}\big|\cdot|f(z)|\,dz\\
&+w^q(B(y,r))^{1/{\beta}-1/s}\cdot\frac{\,C\,}{\lambda}
\sum_{j=1}^\infty\frac{1}{|B(y,2^{j+1}r)|^{1-{\gamma}/n}}\int_{B(y,2^{j+1}r)}\big|b_{B(y,2^{j+1}r)}-b_{B(y,r)}\big|\cdot|f(z)|\,dz\\
&:=J'_5(y,r)+J'_6(y,r).
\end{split}
\end{equation*}
For the term $J'_5(y,r)$, since $w\in A_1$, it follows directly from the $A_1$ condition and the inequality $t\leq \Phi(t)$ that
\begin{equation*}
\begin{split}
J'_5(y,r)&\leq w^q(B(y,r))^{1/{\beta}-1/s}\\
&\times\frac{\,C\,}{\lambda} \sum_{j=1}^\infty\frac{|B(y,2^{j+1}r)|^{\gamma/n}}{w(B(y,2^{j+1}r))}
\int_{B(y,2^{j+1}r)}\big|b(z)-b_{B(y,2^{j+1}r)}\big|\cdot|f(z)|w(z)\,dz\\
&\leq C\cdot w^q(B(y,r))^{1/{\beta}-1/s}\\
&\times\sum_{j=1}^\infty\frac{|B(y,2^{j+1}r)|^{\gamma/n}}{w(B(y,2^{j+1}r))}\int_{B(y,2^{j+1}r)}\big|b(z)-b_{B(y,2^{j+1}r)}\big|
\cdot\Phi\left(\frac{|f(z)|}{\lambda}\right)w(z)\,dz.
\end{split}
\end{equation*}
Furthermore, we use the generalized H\"older's inequality \eqref{Wholder} to obtain
\begin{equation*}
\begin{split}
J'_5(y,r)&\leq C\cdot w^q(B(y,r))^{1/{\beta}-1/s}\\
&\times\sum_{j=1}^\infty\big|B(y,2^{j+1}r)\big|^{\gamma/n}\cdot\big\|b-b_{B(y,2^{j+1}r)}\big\|_{\exp L(w),B(y,2^{j+1}r)}
\bigg\|\Phi\left(\frac{|f|}{\,\lambda\,}\right)\bigg\|_{L\log L(w),B(y,2^{j+1}r)}\\
&\leq C\|b\|_*\sum_{j=1}^\infty\bigg\|\Phi\left(\frac{|f|}{\,\lambda\,}\right)\bigg\|_{L\log L(w),B(y,2^{j+1}r)}
\times\big|B(y,2^{j+1}r)\big|^{\gamma/n}w^q(B(y,r))^{1/{\beta}-1/s}.
\end{split}
\end{equation*}
In the last inequality, we have used the well-known fact that (see \cite{zhang} for instance)
\begin{equation}\label{Jensen}
\big\|b-b_{B}\big\|_{\exp L(w),B}\leq C\|b\|_*,\qquad \mbox{for every ball }B\subset\mathbb R^n.
\end{equation}
It is equivalent to the inequality
\begin{equation*}
\frac{1}{w(B)}\int_B\exp\bigg(\frac{|b(y)-b_B|}{c_0\|b\|_*}\bigg)w(y)\,dy\leq C,
\end{equation*}
which is just a corollary of the well-known John--Nirenberg's inequality (see \cite{john}) and the comparison property of $A_1$ weights. In addition, by the estimate \eqref{wq}
\begin{equation*}
\begin{split}
J'_5(y,r)
&\leq C\|b\|_*\sum_{j=1}^\infty\bigg\|\Phi\left(\frac{|f|}{\,\lambda\,}\right)\bigg\|_{L\log L(w),B(y,2^{j+1}r)}\\
&\times w^q\big(B(y,2^{j+1}r)\big)^{1/{\alpha}-1-1/s}w\big(B(y,2^{j+1}r)\big)
\cdot\frac{w^q(B(y,r))^{1/{\beta}-1/s}}{w^q(B(y,2^{j+1}r))^{1/{\beta}-1/s}}.
\end{split}
\end{equation*}
For the last term $J'_6(y,r)$ we proceed as follows. Using the first part of Lemma \ref{BMO} together with the facts that $w\in A_1$ and $t\leq\Phi(t)=t\cdot(1+\log^+t)$, we deduce that
\begin{equation*}
\begin{split}
J'_6(y,r)&\leq C\cdot w^q(B(y,r))^{1/{\beta}-1/s}
\sum_{j=1}^\infty(j+1)\|b\|_*\cdot\frac{1}{|B(y,2^{j+1}r)|^{1-{\gamma}/n}}\int_{B(y,2^{j+1}r)}\frac{|f(z)|}{\lambda}\,dz\\
&\leq C\cdot w^q(B(y,r))^{1/{\beta}-1/s}
\sum_{j=1}^\infty(j+1)\|b\|_*\cdot\frac{|B(y,2^{j+1}r)|^{\gamma/n}}{w(B(y,2^{j+1}r))}\int_{B(y,2^{j+1}r)}\frac{|f(z)|}{\lambda}\cdot w(z)\,dz\\
&\leq C\|b\|_*\cdot w^q(B(y,r))^{1/{\beta}-1/s}
\sum_{j=1}^\infty(j+1)\cdot\frac{|B(y,2^{j+1}r)|^{\gamma/n}}{w(B(y,2^{j+1}r))}
\int_{B(y,2^{j+1}r)}\Phi\left(\frac{|f(z)|}{\lambda}\right)\cdot w(z)\,dz.
\end{split}
\end{equation*}
Making use of the inequalities \eqref{main esti1} and \eqref{wq}, we further obtain
\begin{equation*}
\begin{split}
J'_6(y,r)&\leq C\|b\|_*\sum_{j=1}^\infty\big(j+1\big)\cdot\bigg\|\Phi\left(\frac{|f|}{\,\lambda\,}\right)\bigg\|_{L\log L(w),B(y,2^{j+1}r)}\\
&\times\big|B(y,2^{j+1}r)\big|^{\gamma/n}w^q(B(y,r))^{1/{\beta}-1/s}\\
&\leq C\|b\|_*\sum_{j=1}^\infty\big(j+1\big)\cdot\bigg\|\Phi\left(\frac{|f|}{\,\lambda\,}\right)\bigg\|_{L\log L(w),B(y,2^{j+1}r)}\\
&\times w^q\big(B(y,2^{j+1}r)\big)^{1/{\alpha}-1-1/s}w\big(B(y,2^{j+1}r)\big)
\cdot\frac{w^q(B(y,r))^{1/{\beta}-1/s}}{w^q(B(y,2^{j+1}r))^{1/{\beta}-1/s}}.
\end{split}
\end{equation*}
Summarizing the above discussions, we conclude that
\begin{align}\label{WJ2yr}
J'_2(y,r)
&\leq C\|b\|_*\sum_{j=1}^\infty\big(j+1\big)\cdot\bigg\|\Phi\left(\frac{|f|}{\,\lambda\,}\right)\bigg\|_{L\log L(w),B(y,2^{j+1}r)}\notag\\
&\times w^q\big(B(y,2^{j+1}r)\big)^{1/{\alpha}-1-1/s}w\big(B(y,2^{j+1}r)\big)
\cdot\frac{w^q(B(y,r))^{1/{\beta}-1/s}}{w^q(B(y,2^{j+1}r))^{1/{\beta}-1/s}}\notag\\
&=C\sum_{j=1}^\infty w^q\big(B(y,2^{j+1}r)\big)^{1/{\alpha}-1-1/s}w\big(B(y,2^{j+1}r)\big)
\bigg\|\Phi\left(\frac{|f|}{\,\lambda\,}\right)\bigg\|_{L\log L(w),B(y,2^{j+1}r)}\notag\\
&\times\big(j+1\big)\cdot\frac{w^q(B(y,r))^{1/{\beta}-1/s}}{w^q(B(y,2^{j+1}r))^{1/{\beta}-1/s}}.
\end{align}
Recall that $w^q\in A_1\subset A_\infty$ with $1<q<\infty$. We can now argue exactly as we did in the estimation of \eqref{psi3} to get (now choose $\delta^*$ in \eqref{compare})
\begin{align}\label{psi4}
\sum_{j=1}^\infty\big(j+1\big)\cdot\frac{w^q(B(y,r))^{1/{\beta}-1/s}}{w^q(B(y,2^{j+1}r))^{1/{\beta}-1/s}}
&\leq C\sum_{j=1}^\infty\big(j+1\big)\cdot\left(\frac{|B(y,r)|}{|B(y,2^{j+1}r)|}\right)^{\delta^\ast(1/{\beta}-1/s)}\notag\\
&= C\sum_{j=1}^\infty\big(j+1\big)\cdot\left(\frac{1}{2^{(j+1)n}}\right)^{\delta^\ast(1/{\beta}-1/s)}\notag\\
&\leq C.
\end{align}
Notice that the exponent $\delta^*{(1/{\beta}-1/s)}$ is positive by our assumption, which guarantees that the last series is convergent.
Therefore by taking the $L^s(\mu)$-norm of both sides of \eqref{Jprime}(with respect to the variable $y$), and then using Minkowski's inequality, \eqref{WJ1yr} and \eqref{WJ2yr}, we have
\begin{equation*}
\begin{split}
&\Big\|w^q(B(y,r))^{1/{\beta}-1/q-1/s}\cdot \Big[w^q\Big(\Big\{x\in B(y,r):\big|[b,I_{\gamma}](f)(x)\big|>\lambda\Big\}\Big)\Big]^{1/q}\Big\|_{L^s(\mu)}\\
&\leq\big\|J'_1(y,r)\big\|_{L^s(\mu)}+\big\|J'_2(y,r)\big\|_{L^s(\mu)}\\
&\leq C\bigg\|w^q(B(y,2r))^{1/{\alpha}-1-1/s}w(B(y,2r))\bigg\|\Phi\left(\frac{|f|}{\,\lambda\,}\right)\bigg\|_{L\log L(w),B(y,2r)}\bigg\|_{L^s(\mu)}\\
&+C\sum_{j=1}^\infty\bigg\|w^q(B(y,2^{j+1}r))^{1/{\alpha}-1-1/s}w(B(y,2^{j+1}r))
\bigg\|\Phi\left(\frac{|f|}{\,\lambda\,}\right)\bigg\|_{L\log L(w),B(y,2^{j+1}r)}\bigg\|_{L^s(\mu)}\\
&\times\big(j+1\big)\cdot\frac{w^q(B(y,r))^{1/{\beta}-1/s}}{w^q(B(y,2^{j+1}r))^{1/{\beta}-1/s}}\\
&\leq C\bigg\|\Phi\left(\frac{|f|}{\,\lambda\,}\right)\bigg\|_{(L\log L,L^s)^{\alpha}(w,w^q;\mu)}\\
&+C\bigg\|\Phi\left(\frac{|f|}{\,\lambda\,}\right)\bigg\|_{(L\log L,L^s)^{\alpha}(w,w^q;\mu)}
\times\sum_{j=1}^\infty\big(j+1\big)\cdot\frac{w^q(B(y,r))^{1/{\beta}-1/s}}{w^q(B(y,2^{j+1}r))^{1/{\beta}-1/s}}\\
&\leq C\bigg\|\Phi\left(\frac{|f|}{\,\lambda\,}\right)\bigg\|_{(L\log L,L^s)^{\alpha}(w,w^q;\mu)},
\end{split}
\end{equation*}
where the last inequality follows from \eqref{psi4}. This completes the proof of Theorem \ref{mainthm:4}.
\end{proof}

Let $b(x)$ be a $BMO$ function on $\mathbb R^n$ and $0<\gamma<n$. The related commutator formed by fractional maximal operator $M_{\gamma}$ and $b$ is given by
\begin{equation*}
[b,M_\gamma](f)(x):=\sup_{B\ni x}\frac{1}{|B|^{1-\gamma/n}}\int_B\big|b(x)-b(y)\big|\cdot|f(y)|\,dy,
\end{equation*}
where the supremum is taken over all balls $B$ containing $x$. Obviously, $[b,M_\gamma]$ is a sublinear operator. It should be pointed out that $[b,M_{\gamma}](f)$ can be controlled pointwise by the expression given below. For any $0<\gamma<n$, $x\in\mathbb R^n$ and $r>0$, we have
\begin{equation*}
\begin{split}
\int_{\mathbb R^n}\big|b(x)-b(y)\big|\cdot\frac{|f(y)|}{|x-y|^{n-\gamma}}\,dy&\ge\int_{|y-x|<r}\frac{|b(x)-b(y)|\cdot|f(y)|}{|x-y|^{n-\gamma}}\,dy\\
&\ge \frac{1}{r^{n-\gamma}}\int_{|y-x|<r}|b(x)-b(y)|\cdot|f(y)|\,dy.
\end{split}
\end{equation*}
Taking the supremum for all $r>0$ on both sides of the above inequality, we get
\begin{equation}\label{dominate3}
\int_{\mathbb R^n}\big|b(x)-b(y)\big|\cdot\frac{|f(y)|}{|x-y|^{n-\gamma}}\,dy\geq C\cdot[b,M_{\gamma}](f)(x),\quad\mbox{for all}\; x\in\mathbb R^n,
\end{equation}
which is our desired result. Moreover, on the commutator $[b,M_\gamma]$ of the fractional maximal operator $M_{\gamma}$, we also have the following result:
\begin{theorem}[\cite{lu}]\label{bm}
Let $0<\gamma<n$, $1<p<n/{\gamma}$, $1/q=1/p-{\gamma}/n$ and $w\in A_{p,q}$. Suppose that $b\in BMO(\mathbb R^n)$, then the sublinear operator $[b,M_{\gamma}]$ is bounded from $L^p(w^p)$ to $L^q(w^q)$.
\end{theorem}
Taking into account \eqref{dominate3} and Theorem \ref{bm}, and then using the same arguments as in the proof of Theorem $\ref{mainthm:3}$, we know that the conclusion of Theorem $\ref{mainthm:3}$ still hold for the sublinear operator $[b,M_{\gamma}]$.

\begin{corollary}
Let $0<\gamma<n$, $1<p<n/{\gamma}$, $1/q=1/p-{\gamma}/n$ and $w\in A_{p,q}$. Assume that $p\leq\alpha<\beta<s\leq\infty$, $\mu\in\Delta_2$ and $b\in BMO(\mathbb R^n)$, then the sublinear operator $[b,M_{\gamma}]$ is bounded from $(L^p,L^s)^{\alpha}(w^p,w^q;\mu)$ into $(L^q,L^s)^{\beta}(w^q;\mu)$ with $1/{\beta}=1/{\alpha}-{\gamma}/n$.
\end{corollary}

Let $\mathcal L^{-\gamma/2}$ be the generalized fractional integrals of $\mathcal L$ for $0<\gamma<n$, and let $b$ be a locally integrable function on $\mathbb R^n$. The generalized commutator generated by $b$ and $\mathcal L^{-\gamma/2}$ is defined as follows.
\begin{equation}
\big[b,\mathcal L^{-\gamma/2}\big]f(x):=b(x)\mathcal L^{-\gamma/2}(f)(x)-\mathcal L^{-\gamma/2}(bf)(x).
\end{equation}
By the kernel estimate \eqref{kernelk},
\begin{equation*}
\begin{split}
\big|\big[b,\mathcal L^{-\gamma/2}\big]f(x)\big|&=\bigg|\int_{\mathbb R^n}\big[b(x)-b(y)\big]\mathcal K_\gamma(x,y)f(y)\,dy\bigg|\\
&\leq\int_{\mathbb R^n}\big|b(x)-b(y)\big||\mathcal K_\gamma(x,y)||f(y)|\,dy\\
&\leq C\int_{\mathbb R^n}\big|b(x)-b(y)\big|\cdot\frac{|f(y)|}{|x-y|^{n-\gamma}}\,dy.
\end{split}
\end{equation*}
In 2008, Auscher and Martell \cite{auscher} considered the weighted estimate for $\big[b,\mathcal L^{-\gamma/2}\big]$ and obtained the following result (see also \cite{cruz}).

\begin{theorem}[\cite{auscher}]
Let $0<\gamma<n$, $1<p<n/{\gamma}$, $1/q=1/p-{\gamma}/n$ and $w\in A_{p,q}$. Suppose that $b\in BMO(\mathbb R^n)$, then the generalized commutator $\big[b,\mathcal L^{-\gamma/2}\big]$ is bounded from $L^p(w^p)$ to $L^q(w^q)$.
\end{theorem}

Hence, as a direct consequence of the above results, we can also obtain

\begin{corollary}
Let $0<\gamma<n$, $1<p<n/{\gamma}$, $1/q=1/p-{\gamma}/n$ and $w\in A_{p,q}$. Assume that $p\leq\alpha<\beta<s\leq\infty$, $\mu\in\Delta_2$ and $b\in BMO(\mathbb R^n)$, then the generalized commutator $\big[b,\mathcal L^{-\gamma/2}\big]$ is bounded from $(L^p,L^s)^{\alpha}(w^p,w^q;\mu)$ into $(L^q,L^s)^{\beta}(w^q;\mu)$ with $1/{\beta}=1/{\alpha}-{\gamma}/n$.
\end{corollary}

\section{Proof of Theorem \ref{mainthm:end}}
This section is concerned with the proofs of Theorem \ref{mainthm:end} and the corresponding result for generalized fractional integrals $\mathcal L^{-\gamma/2}$.
\begin{proof}[Proof of Theorem $\ref{mainthm:end}$]
Let $1<p\leq\alpha<s\leq\infty$ and $f\in(L^p,L^s)^{\alpha}(w^p,w^q;\mu)$ with $w\in A_{p,q}$ and $\mu\in\Delta_2$. For any fixed ball $B=B(y,r)$ in $\mathbb R^n$, we are going to estimate the following expression
\begin{equation}\label{end1.1}
\frac{1}{|B(y,r)|}\int_{B(y,r)}\big|I_{\gamma}f(x)-(I_{\gamma}f)_{B(y,r)}\big|\,dx.
\end{equation}
Decompose $f$ as $f=f_1+f_2$, where $f_1=f\cdot\chi_{4B}$, $f_2=f\cdot\chi_{(4B)^c}$, $4B=B(y,4r)$. By the linearity of the fractional integral operator $I_{\gamma}$, the above expression \eqref{end1.1} can be divided into two parts. That is,
\begin{equation*}
\begin{split}
&\frac{1}{|B(y,r)|}\int_{B(y,r)}\big|I_{\gamma}f(x)-(I_{\gamma}f)_{B(y,r)}\big|\,dx\\
&\leq \frac{1}{|B(y,r)|}\int_{B(y,r)}\big|I_{\gamma}f_1(x)-(I_{\gamma}f_1)_{B(y,r)}\big|\,dx
+\frac{1}{|B(y,r)|}\int_{B(y,r)}\big|I_{\gamma}f_2(x)-(I_{\gamma}f_2)_{B(y,r)}\big|\,dx\\
&:=I(y,r)+II(y,r).
\end{split}
\end{equation*}
Let us first consider the term $I(y,r)$. Applying the weighted $(L^p,L^q)$-boundedness of $I_{\gamma}$ (see Theorem \ref{strong}) and H\"older's inequality, we obtain
\begin{equation*}
\begin{split}
I(y,r)&\leq\frac{2}{|B(y,r)|}\int_{B(y,r)}|I_{\gamma}f_1(x)|\,dx\\
&\leq\frac{2}{|B(y,r)|}\bigg(\int_{B(y,r)}|I_{\gamma}f_1(x)|^qw^q(x)\,dx\bigg)^{1/q}\bigg(\int_{B(y,r)} w(x)^{-q'}dx\bigg)^{1/{q'}}\\
&\leq\frac{C}{|B(y,r)|}\bigg(\int_{B(y,4r)}|f(x)|^pw^p(x)\,dx\bigg)^{1/p}\bigg(\int_{B(y,r)}w(x)^{-q'}dx\bigg)^{1/{q'}}.
\end{split}
\end{equation*}
Since $w$ is a weight in the class $A_{p,q}$, one has $w^q\in A_q\subset A_{\infty}$ by Lemma \ref{relation}$(i)$. By definition, it reads
\begin{equation*}
\bigg(\frac1{|B(y,r)|}\int_{B(y,r)}w^q(x)\,dx\bigg)^{1/q}\bigg(\frac1{|B(y,r)|}\int_{B(y,r)}\big[w^q(x)\big]^{-q'/q}\,dx\bigg)^{1/{q'}}\leq C,
\end{equation*}
which implies
\begin{equation}\label{end1.2}
\bigg(\int_{B(y,r)}w(x)^{-q'}dx\bigg)^{1/{q'}}\leq C\cdot\frac{|B(y,r)|}{w^q(B(y,r))^{1/q}}.
\end{equation}
Since $w^q\in A_q\subset A_{\infty}$, then $w^q\in\Delta_2$. Using the inequalities \eqref{end1.2} and \eqref{weights}, we have
\begin{equation}\label{IW}
\begin{split}
I(y,r)&\leq C\bigg(\int_{B(y,4r)}|f(x)|^pw^p(x)\,dx\bigg)^{1/p}\cdot\frac{1}{w^q(B(y,r))^{1/q}}\\
&\leq C\bigg(\int_{B(y,4r)}|f(x)|^pw^p(x)\,dx\bigg)^{1/p}\cdot\frac{1}{w^q(B(y,4r))^{1/q}}\\
&=C\cdot w^q(B(y,4r))^{1/{\alpha}-1/p-1/s}\big\|f\cdot\chi_{B(y,4r)}\big\|_{L^p(w^p)},
\end{split}
\end{equation}
where in the last equality we have used the hypothesis $1/{s}=1/{\alpha}-{\gamma}/n$ and $1/q=1/p-{\gamma}/n$.
We now turn to estimate the second term $II(y,r)$. For any $x\in B(y,r)$,
\begin{equation*}
\begin{split}
\big|I_{\gamma}f_2(x)-(I_{\gamma}f_2)_{B(y,r)}\big|&=\bigg|\frac{1}{|B(y,r)|}\int_{B(y,r)}\big[I_{\gamma}f_2(x)-I_{\gamma}f_2(z)\big]\,dz\bigg|\\
&=\bigg|\frac{1}{|B(y,r)|}\int_{B(y,r)}
\bigg\{\int_{B(y,4r)^c}\bigg[\frac{1}{|x-\zeta|^{n-\gamma}}-\frac{1}{|z-\zeta|^{n-\gamma}}\bigg]f(\zeta)\,d\zeta\bigg\}dz\bigg|\\
&\leq\frac{1}{|B(y,r)|}\int_{B(y,r)}
\bigg\{\int_{B(y,4r)^c}\bigg|\frac{1}{|x-\zeta|^{n-\gamma}}-\frac{1}{|z-\zeta|^{n-\gamma}}\bigg|\cdot|f(\zeta)|\,d\zeta\bigg\}dz.
\end{split}
\end{equation*}
Since both $x$ and $z$ are in $B(y,r)$, $\zeta\in B(y,4r)^c$, by a purely geometric observation, we must have $|x-\zeta|\geq 2|x-z|$. This fact along with the mean value theorem yields
\begin{align}\label{average}
\big|I_{\gamma}f_2(x)-(I_{\gamma}f_2)_{B(y,r)}\big|
&\leq\frac{C}{|B(y,r)|}\int_{B(y,r)}\bigg\{\int_{B(y,4r)^c}\frac{|x-z|}{|x-\zeta|^{n-\gamma+1}}\cdot|f(\zeta)|\,d\zeta\bigg\}dz\notag\\
&\leq C\int_{B(y,4r)^c}\frac{r}{|\zeta-y|^{n-\gamma+1}}\cdot|f(\zeta)|\,d\zeta\notag\\
&=C\sum_{j=2}^\infty\int_{B(y,2^{j+1}r)\backslash B(y,2^jr)}\frac{r}{|\zeta-y|^{n-\gamma+1}}\cdot|f(\zeta)|\,d\zeta\notag\\
&\leq C\sum_{j=2}^\infty\frac{1}{2^j}\cdot\frac{1}{|B(y,2^{j+1}r)|^{1-{\gamma}/n}}\int_{B(y,2^{j+1}r)}|f(\zeta)|\,d\zeta.
\end{align}
Furthermore, by using H\"older's inequality and $A_{p,q}$ condition on $w$, the last expression in \eqref{average} can be estimated as follows:
\begin{align}\label{end1.3}
&\frac{1}{|B(y,2^{j+1}r)|^{1-{\gamma}/n}}\int_{B(y,2^{j+1}r)}|f(\zeta)|\,d\zeta\notag\\ &\leq\frac{1}{|B(y,2^{j+1}r)|^{1-{\gamma}/n}}\bigg(\int_{B(y,2^{j+1}r)}\big|f(\zeta)\big|^pw^p(\zeta)\,d\zeta\bigg)^{1/p}
\bigg(\int_{B(y,2^{j+1}r)}w(\zeta)^{-p'}d\zeta\bigg)^{1/{p'}}\notag\\
&\leq C\bigg(\int_{B(y,2^{j+1}r)}\big|f(\zeta)\big|^pw^p(\zeta)\,d\zeta\bigg)^{1/p}
\cdot\frac{1}{w^q(B(y,2^{j+1}r))^{1/q}}\notag\\
&=C\cdot w^q\big(B(y,2^{j+1}r)\big)^{1/{\alpha}-1/p-1/s}\big\|f\cdot\chi_{B(y,2^{j+1}r)}\big\|_{L^p(w^p)},
\end{align}
where in the last equality we have used the fact that $1/{\alpha}-1/p-1/s=-1/q$ again.
From this, it readily follows that for any $x\in B(y,r)$,
\begin{equation*}
\big|I_{\gamma}f_2(x)-(I_{\gamma}f_2)_{B(y,r)}\big|
\leq C\sum_{j=2}^\infty\frac{1}{2^j}
\cdot w^q\big(B(y,2^{j+1}r)\big)^{1/{\alpha}-1/p-1/s}\big\|f\cdot\chi_{B(y,2^{j+1}r)}\big\|_{L^p(w^p)}.
\end{equation*}
Consequently,
\begin{align}\label{IIW}
II(y,r)&=\frac{1}{|B(y,r)|}\int_{B(y,r)}\big|I_{\gamma}f_2(x)-(I_{\gamma}f_2)_{B(y,r)}\big|\,dx\notag\\
&\leq C\sum_{j=2}^\infty\frac{1}{2^j}\cdot w^q\big(B(y,2^{j+1}r)\big)^{1/{\alpha}-1/p-1/s}\big\|f\cdot\chi_{B(y,2^{j+1}r)}\big\|_{L^p(w^p)}.
\end{align}
Therefore by taking the $L^s(\mu)$-norm of \eqref{end1.1}(with respect to the variable $y$), and then using Minkowski's inequality, \eqref{IW} and \eqref{IIW}, we get
\begin{equation*}
\begin{split}
&\bigg\|\frac{1}{|B(y,r)|}\int_{B(y,r)}\big|I_{\gamma}f(x)-(I_{\gamma}f)_{B(y,r)}\big|\,dx\bigg\|_{L^s(\mu)}\\
&\leq\big\|I(y,r)\big\|_{L^s(\mu)}+\big\|II(y,r)\big\|_{L^s(\mu)}\\
&\leq C\Big\|w^q(B(y,4r))^{1/{\alpha}-1/p-1/s}\big\|f\cdot\chi_{B(y,4r)}\big\|_{L^p(w^p)}\Big\|_{L^s(\mu)}\\
&+C\sum_{j=2}^\infty\frac{1}{2^j}\cdot\Big\|w^q(B(y,2^{j+1}r))^{1/{\alpha}-1/p-1/s}\big\|f\cdot\chi_{B(y,2^{j+1}r)}\big\|_{L^p(w^p)}\Big\|_{L^s(\mu)}\\
&\leq C\big\|f\big\|_{(L^p,L^s)^{\alpha}(w^p,w^q;\mu)}+C\big\|f\big\|_{(L^p,L^s)^{\alpha}(w^p,w^q;\mu)}
\times\sum_{j=2}^\infty\frac{1}{2^j}\\
&\leq C\big\|f\big\|_{(L^p,L^s)^{\alpha}(w^p,w^q;\mu)}.
\end{split}
\end{equation*}
By taking the supremum over all $r>0$, we are done.
\end{proof}
For any $f\in L^p(\mathbb R^n)$, $1\le p<\infty$, Martell \cite{martell} defined a kind of \emph{sharp} maximal function $M^{\#}_{\mathcal L} f$ associated with the semigroup $\big\{e^{-t\mathcal L}\big\}_{t>0}$ by the following expression
\begin{equation*}
M^{\#}_{\mathcal L}f(x):=\sup_{x\in B}\frac{1}{|B|}\int_B\big|f(y)-e^{-t_B{\mathcal L}}f(y)\big|\,dy,
\end{equation*}
where $t_B=r_B^2$ and $r_B$ is the radius of the ball $B$. We say that $f\in BMO_{\mathcal L}$ if the \emph{sharp} maximal function $M^{\#}_{\mathcal L} f\in L^\infty(\mathbb R^n)$, and we define $\|f\|_{BMO_{\mathcal L}}=\big\|M^{\#}_{\mathcal L} f\big\|_{L^\infty}$. Inspired by this notion and Theorem \ref{mainthm:end}, a natural question for the generalized fractional integrals $\mathcal L^{-\gamma/2}$ is the following: can we get any result corresponding to  Theorem \ref{mainthm:end} for the limiting case $\beta=s$? For this purpose, we need to introduce the following $BMO$-type space associated with the semigroup $\big\{e^{-t\mathcal L}\big\}_{t>0}$.
\begin{defn}
Let $1\leq s\leq\infty$ and $\mu\in\Delta_2$. We define the space $(BMO_{\mathcal L},L^s)(\mu)$ as the set of all locally integrable functions $f$ satisfying $\|f\|_{***}<\infty$, where
\begin{equation}\label{BMOw}
\|f\|_{***}:=\sup_{r>0}\bigg\|\frac{1}{|B(y,r)|}\int_{B(y,r)}\big|f(x)-e^{-r^2{\mathcal L}}f(x)\big|\,dx\bigg\|_{L^s(\mu)},
\end{equation}
and the $L^s(\mu)$-norm is taken with respect to the variable $y$.
\end{defn}
Based on the above notion, we can prove the following result.
\begin{theorem}\label{mainthm:end2}
Let $0<\gamma<n$, $1<p<n/{\gamma}$, $1/q=1/p-{\gamma}/n$, and let $w\in A_{p,q}$ and $\mu\in\Delta_2$. If $p\leq\alpha<s\leq\infty$ and $1/{s}=1/{\alpha}-{\gamma}/n$, then the generalized fractional integrals $\mathcal L^{-\gamma/2}$ is bounded from $(L^p,L^s)^{\alpha}(w^p,w^q;\mu)$ into $(BMO_{\mathcal L},L^s)(\mu)$.
\end{theorem}
\begin{proof}
Let $1<p\leq\alpha<s\leq\infty$ and $f\in(L^p,L^s)^{\alpha}(w^p,w^q;\mu)$ with $w\in A_{p,q}$ and $\mu\in\Delta_2$. In this situation, for any given ball $B=B(y,r)$ in $\mathbb R^n$, we need to consider the following expression
\begin{equation}\label{end2.1}
\frac{1}{|B(y,r)|}\int_{B(y,r)}\big|\mathcal L^{-\gamma/2}f(x)-e^{-r^2{\mathcal L}}\big(\mathcal L^{-\gamma/2}f\big)(x)\big|\,dx.
\end{equation}
Decompose $f$ as $f=f_1+f_2$, where $f_1=f\cdot\chi_{4B}$, $f_2=f\cdot\chi_{(4B)^c}$, $4B=B(y,4r)$. Similarly, the above expression \eqref{end2.1} can be divided into three parts. That is,
\begin{equation*}
\begin{split}
&\frac{1}{|B(y,r)|}\int_{B(y,r)}\big|\mathcal L^{-\gamma/2}f(x)-e^{-r^2{\mathcal L}}\big(\mathcal L^{-\gamma/2}f\big)(x)\big|\,dx\\
&\leq \frac{1}{|B(y,r)|}\int_{B(y,r)}\big|\mathcal L^{-\gamma/2}f_1(x)\big|\,dx
+\frac{1}{|B(y,r)|}\int_{B(y,r)}\big|e^{-r^2{\mathcal L}}\big(\mathcal L^{-\gamma/2}f_1\big)(x)\big|\,dx\\
&+\frac{1}{|B(y,r)|}\int_{B(y,r)}\big|\mathcal L^{-\gamma/2}f_2(x)-e^{-r^2{\mathcal L}}\big(\mathcal L^{-\gamma/2}f_2\big)(x)\big|\,dx\\
&:=I'(y,r)+II'(y,r)+III'(y,r).
\end{split}
\end{equation*}
First let us consider the term $I'(y,r)$. By \eqref{dominate1} and Theorem \ref{strong}, we know that the generalized fractional integrals $\mathcal L^{-\gamma/2}$ is also bounded from $L^p(w^p)$ to $L^q(w^q)$ whenever $w\in A_{p,q}$. This fact along with H\"older's inequality implies
\begin{equation*}
\begin{split}
I'(y,r)&\leq\frac{1}{|B(y,r)|}\bigg(\int_{B(y,r)}|\mathcal L^{-\gamma/2}f_1(x)|^qw^q(x)\,dx\bigg)^{1/q}\bigg(\int_{B(y,r)} w(x)^{-q'}dx\bigg)^{1/{q'}}\\
&\leq\frac{C}{|B(y,r)|}\bigg(\int_{B(y,4r)}|f(x)|^pw^p(x)\,dx\bigg)^{1/p}\bigg(\int_{B(y,r)}w(x)^{-q'}dx\bigg)^{1/{q'}}.
\end{split}
\end{equation*}
We now proceed exactly as we did in the proof of Theorem \ref{mainthm:end}, then
\begin{equation}\label{IpW}
I'(y,r)\leq C\cdot w^q(B(y,4r))^{1/{\alpha}-1/p-1/s}\big\|f\cdot\chi_{B(y,4r)}\big\|_{L^p(w^p)}.
\end{equation}
For the term $II'(y,r)$, since the kernel of $e^{-r^2\mathcal L}$ is $p_{r^2}(x,z)$, then we may write
\begin{equation*}
\begin{split}
II'(y,r)\leq&\frac{1}{|B(y,r)|}\int_{B(y,r)}\bigg\{\int_{\mathbb R^n}\big|p_{r^2}(x,z)\cdot\mathcal L^{-\gamma/2}f_1(z)\big|dz\bigg\}dx\\
=&\frac{1}{|B(y,r)|}\int_{B(y,r)}\bigg\{\int_{B(y,4r)}\big|p_{r^2}(x,z)\cdot\mathcal L^{-\gamma/2}f_1(z)\big|dz\bigg\}dx\\
&+\sum_{j=2}^\infty\frac{1}{|B(y,r)|}\int_{B(y,r)}\bigg\{\int_{B(y,2^{j+1}r)\backslash B(y,2^jr)}
\big|p_{r^2}(x,z)\cdot\mathcal L^{-\gamma/2}f_1(z)\big|dz\bigg\}dx\\
=&II'_{(1)}(y,r)+II'_{(2)}(y,r).
\end{split}
\end{equation*}
For any $x\in B(y,r)$ and $z\in B(y,4r)$, by (\ref{G}), we have $\big|p_{r^2}(x,z)\big|\le C\cdot(r^2)^{-n/2}$. Thus,
\begin{equation*}
\begin{split}
II'_{(1)}(y,r)&\leq\frac{C}{|B(y,r)|}\int_{B(y,r)}\bigg\{\int_{B(y,4r)}\frac{1}{(r^2)^{n/2}}\big|\mathcal L^{-\gamma/2}f_1(z)\big|dz\bigg\}dx\\
&\leq\frac{C}{|B(y,4r)|}\int_{B(y,4r)}\big|\mathcal L^{-\gamma/2}f_1(z)\big|dz.
\end{split}
\end{equation*}
On the other hand, note that for any $x\in B(y,r)$, $z\in B(y,4r)^c$, then $|z-y|\approx|z-x|$. In this case, by using (\ref{G}) again, we get $\big|p_{r^2}(x,z)\big|\le C\cdot\frac{(r^2)^{n/2}}{|x-z|^{2n}}$. Hence,
\begin{equation*}
\begin{split}
II'_{(2)}(y,r)&\leq C\sum_{j=2}^\infty\frac{1}{|B(y,r)|}\int_{B(y,r)}
\bigg\{\int_{B(y,2^{j+1}r)\backslash B(y,2^jr)}\frac{r^n}{|x-z|^{2n}}\big|\mathcal L^{-\gamma/2}f_1(z)\big|dz\bigg\}dx\\
&\leq C\sum_{j=2}^\infty\int_{B(y,2^{j+1}r)\backslash B(y,2^jr)}\frac{r^n}{|y-z|^{2n}}\big|\mathcal L^{-\gamma/2}f_1(z)\big|dz\\
&\leq C\sum_{j=2}^\infty\frac{1}{(2^j)^n}\frac{1}{|B(y,2^{j+1}r)|}\int_{B(y,2^{j+1}r)}\big|\mathcal L^{-\gamma/2}f_1(z)\big|dz.
\end{split}
\end{equation*}
Summing up the above estimates for $II'_{(1)}(y,r)$ and $II'_{(2)}(y,r)$, we get
\begin{equation*}
II'(y,r)\leq C\sum_{j=1}^\infty\frac{1}{(2^j)^n}\frac{1}{|B(y,2^{j+1}r)|}\int_{B(y,2^{j+1}r)}\big|\mathcal L^{-\gamma/2}f_1(z)\big|dz.
\end{equation*}
Furthermore, using H\"older's inequality and invoking the weighted boundedness of $\mathcal L^{-\gamma/2}$ mentioned above, we can deduce that
\begin{equation*}
\begin{split}
II'(y,r)&\leq C\sum_{j=1}^\infty\frac{1}{(2^j)^n}\frac{1}{|B(y,2^{j+1}r)|}\\
&\times\bigg(\int_{B(y,2^{j+1}r)}|\mathcal L^{-\gamma/2}f_1(z)|^qw^q(z)\,dz\bigg)^{1/q}\bigg(\int_{B(y,2^{j+1}r)} w(z)^{-q'}dz\bigg)^{1/{q'}}\\
&\leq C\sum_{j=1}^\infty\frac{1}{(2^j)^n}\frac{1}{|B(y,2^{j+1}r)|}\\
&\times\bigg(\int_{B(y,2^{j+1}r)}|f(z)|^pw^p(z)\,dz\bigg)^{1/p}\bigg(\int_{B(y,2^{j+1}r)}w(z)^{-q'}dz\bigg)^{1/{q'}}.
\end{split}
\end{equation*}
Again, arguing as in the proof of Theorem \ref{mainthm:end}, we can also show that
\begin{equation}\label{IIpW}
II'(y,r)\leq C\sum_{j=1}^\infty\frac{1}{(2^j)^n}\cdot w^q\big(B(y,2^{j+1}r)\big)^{1/{\alpha}-1/p-1/s}\big\|f\cdot\chi_{B(y,2^{j+1}r)}\big\|_{L^p(w^p)}.
\end{equation}
In order to estimate the last term $III'(y,r)$, we need the following key lemma given in \cite{deng} (see also \cite{duong1}).
\begin{lemma}
For $0<\gamma<n$, the difference operator $(I-e^{-t\mathcal L})\mathcal L^{-\gamma/2}$ has an associated kernel $\widetilde{\mathcal K}_{\gamma,t}(x,z)$ which satisfies the following estimate:
\begin{equation}\label{widek}
\big|\widetilde{\mathcal K}_{\gamma,t}(x,z)\big|\le\frac{C}{|x-z|^{n-\gamma}}\cdot\frac{t}{|x-z|^2}.
\end{equation}
\end{lemma}
Let us return to the proof of $III'(y,r)$. By the above kernel estimate (\ref{widek}), we have
\begin{equation*}
\begin{split}
III'(y,r)&=\frac{1}{|B(y,r)|}\int_{B(y,r)}\big|(I-e^{-r^2\mathcal L})\mathcal L^{-\gamma/2}(f_2)(x)\big|\,dx\\
&\leq\frac{1}{|B(y,r)|}\int_{B(y,r)}\bigg\{\int_{B(y,4r)^c}\big|\widetilde{\mathcal K}_{\gamma,r^2}(x,z)f(z)\big|\,dz\bigg\}dx\\
&\leq\frac{C}{|B(y,r)|}\int_{B(y,r)}\bigg\{\int_{B(y,4r)^c}\frac{1}{|x-z|^{n-\gamma}}\cdot\frac{r^2}{|x-z|^2}|f(z)|\,dz\bigg\}dx\\
&\le C\int_{B(y,4r)^c}\frac{1}{|y-z|^{n-\gamma}}\cdot\frac{r^2}{|y-z|^2}|f(z)|\,dz,
\end{split}
\end{equation*}
where the last inequality is due to $|x-z|\approx|y-z|$ when $x\in B(y,r)$ and $z\in B(y,4r)^c$. Hence,
\begin{align*}
III'(y,r)
&\leq C\sum_{j=2}^\infty\frac{1}{(2^j)^2}\cdot\frac{1}{|B(y,2^{j+1}r)|^{1-{\gamma}/n}}\int_{B(y,2^{j+1}r)}|f(z)|\,dz.
\end{align*}
Moreover, in view of the estimate \eqref{end1.3}, we obtain
\begin{equation}\label{IIIpW}
III'(y,r)\leq C\sum_{j=2}^\infty\frac{1}{(2^j)^2}\cdot w^q\big(B(y,2^{j+1}r)\big)^{1/{\alpha}-1/p-1/s}\big\|f\cdot\chi_{B(y,2^{j+1}r)}\big\|_{L^p(w^p)}.
\end{equation}
Therefore by taking the $L^s(\mu)$-norm of \eqref{end2.1}(with respect to the variable $y$), and then using Minkowski's inequality, \eqref{IpW}, \eqref{IIpW} and \eqref{IIIpW}, we get
\begin{equation*}
\begin{split}
&\bigg\|\frac{1}{|B(y,r)|}\int_{B(y,r)}\big|\mathcal L^{-\gamma/2}f(x)-e^{-r^2{\mathcal L}}\big(\mathcal L^{-\gamma/2}f\big)(x)\big|\,dx\bigg\|_{L^s(\mu)}\\
&\leq\big\|I'(y,r)\big\|_{L^s(\mu)}+\big\|II'(y,r)\big\|_{L^s(\mu)}+\big\|III'(y,r)\big\|_{L^s(\mu)}\\
&\leq C\Big\|w^q(B(y,4r))^{1/{\alpha}-1/p-1/s}\big\|f\cdot\chi_{B(y,4r)}\big\|_{L^p(w^p)}\Big\|_{L^s(\mu)}\\
&+C\sum_{j=1}^\infty\frac{1}{(2^j)^n}\cdot\Big\|w^q(B(y,2^{j+1}r))^{1/{\alpha}-1/p-1/s}\big\|f\cdot\chi_{B(y,2^{j+1}r)}\big\|_{L^p(w^p)}\Big\|_{L^s(\mu)}\\
&+C\sum_{j=2}^\infty\frac{1}{(2^j)^2}\cdot\Big\|w^q(B(y,2^{j+1}r))^{1/{\alpha}-1/p-1/s}\big\|f\cdot\chi_{B(y,2^{j+1}r)}\big\|_{L^p(w^p)}\Big\|_{L^s(\mu)}\\
&\leq C\big\|f\big\|_{(L^p,L^s)^{\alpha}(w^p,w^q;\mu)}+C\big\|f\big\|_{(L^p,L^s)^{\alpha}(w^p,w^q;\mu)}
\times\sum_{j=1}^\infty\frac{1}{(2^j)^n}\\
&+C\big\|f\big\|_{(L^p,L^s)^{\alpha}(w^p,w^q;\mu)}
\times\sum_{j=2}^\infty\frac{1}{(2^j)^2}\\
&\leq C\big\|f\big\|_{(L^p,L^s)^{\alpha}(w^p,w^q;\mu)}.
\end{split}
\end{equation*}
We end the proof by taking the supremum over all $r>0$.
\end{proof}

\section{Some results on two-weight problems}

In the last section, we consider related problems about two-weight, weak type norm inequalities for $I_{\gamma}$ and $[b,I_{\gamma}]$ on weighted amalgam spaces. In \cite{cruz2}, Cruz-Uribe and P\'erez considered the problem of finding sufficient conditions on a pair of weights $(w,\nu)$ which ensure the boundedness of the operator $I_{\gamma}$ from $L^p(\nu)$ to $WL^p(w)$, where $1<p<\infty$. They gave a sufficient $A_p$-type condition (see \eqref{assump1.1} below), and proved a two-weight, weak-type $(p,p)$ inequality for $I_{\gamma}$(see also \cite{cruz3} for another, more simpler proof), which solved a problem posed by Sawyer and Wheeden in \cite{sawyer}.
\begin{theorem}[\cite{cruz2,cruz3}]\label{Two1}
Let $0<\gamma<n$ and $1<p<\infty$. Given a pair of weights $(w,\nu)$, suppose that for some $r>1$ and for all cubes $Q$ in $\mathbb R^n$,
\begin{equation}\label{assump1.1}
\big|Q\big|^{\gamma/n}\cdot\left(\frac{1}{|Q|}\int_Q w(x)^r\,dx\right)^{1/{(rp)}}\left(\frac{1}{|Q|}\int_Q \nu(x)^{-p'/p}\,dx\right)^{1/{p'}}\leq C<\infty.
\end{equation}
Then the fractional integral operator $I_\gamma$ satisfies the weak-type $(p,p)$ inequality
\begin{equation}\label{assump1.2}
w\big(\big\{x\in\mathbb R^n:\big|I_\gamma f(x)\big|>\sigma\big\}\big)
\leq \frac{C}{\sigma^p}\int_{\mathbb R^n}|f(x)|^p \nu(x)\,dx,\quad\mbox{for any }~\sigma>0,
\end{equation}
where $C$ does not depend on $f$ and $\sigma>0$.
\end{theorem}

Moreover, in \cite{li}, Li improved this result by replacing the ``power bump" in \eqref{assump1.1} by a smaller ``Orlicz bump". On the other hand, in \cite{liu}, Liu and Lu obtained a sufficient $A_p$-type condition for the commutator $[b,I_{\gamma}]$ to satisfy the two-weight weak type $(p,p)$ inequality, where $1<p<\infty$. That condition is an $A_p$-type condition in the scale of Orlicz spaces (see \eqref{assump2.1} below).
\begin{theorem}[\cite{liu}]\label{Two2}
Let $0<\gamma<n$, $1<p<\infty$ and $b\in BMO(\mathbb R^n)$. Given a pair of weights $(w,\nu)$, suppose that for some $r>1$ and for all cubes $Q$ in $\mathbb R^n$,
\begin{equation}\label{assump2.1}
\big|Q\big|^{\gamma/n}\cdot\left(\frac{1}{|Q|}\int_Q w(x)^r\,dx\right)^{1/{(rp)}}\big\|\nu^{-1/p}\big\|_{\mathcal A,Q}\leq C<\infty,
\end{equation}
where $\mathcal A(t)=t^{p'}\cdot(1+\log^+t)^{p'}$. Then the linear commutator $[b,I_\gamma]$ satisfies the weak-type $(p,p)$ inequality
\begin{equation}\label{assump2.2}
w\big(\big\{x\in\mathbb R^n:\big|[b,I_\gamma](f)(x)\big|>\sigma\big\}\big)
\leq \frac{C}{\sigma^p}\int_{\mathbb R^n}|f(x)|^p \nu(x)\,dx,\quad\mbox{for any }~\sigma>0,
\end{equation}
where $C$ does not depend on $f$ and $\sigma>0$.
\end{theorem}
Here and in what follows, all cubes are assumed to have their sides parallel to the coordinate axes, $Q(y,\ell)$ will denote the cube centered at $y$ and has side length $\ell$. For any cube $Q(y,\ell)$ and any $\lambda>0$, $\lambda Q$ stands for the cube concentric with $Q$ and having side length $\lambda$ times as long, i.e., $\lambda Q:=Q(y,\lambda\ell)$. We now extend the results mentioned above to the weighted amalgam spaces.

\begin{theorem}\label{mainthm:5}
Let $0<\gamma<n$, $1<p\leq\alpha<s\leq\infty$ and $\mu\in\Delta_2$. Given a pair of weights $(w,\nu)$, suppose that for some $r>1$ and for all cubes $Q$ in $\mathbb R^n$, \eqref{assump1.1} holds.
If $w\in \Delta_2$, then the fractional integral operator $I_{\gamma}$ is bounded from $(L^p,L^s)^{\alpha}(\nu,w;\mu)$ into $(WL^p,L^s)^{\alpha}(w;\mu)$.
\end{theorem}

\begin{theorem}\label{mainthm:6}
Let $0<\gamma<n$, $1<p\leq\alpha<s\leq\infty$, $\mu\in\Delta_2$ and $b\in BMO(\mathbb R^n)$. Given a pair of weights $(w,\nu)$, suppose that for some $r>1$ and for all cubes $Q$ in $\mathbb R^n$, \eqref{assump2.1} holds. If $w\in A_\infty$, then the linear commutator $[b,I_{\gamma}]$ is bounded from $(L^p,L^s)^{\alpha}(\nu,w;\mu)$ into $(WL^p,L^s)^{\alpha}(w;\mu)$.
\end{theorem}

\begin{proof}[Proof of Theorem $\ref{mainthm:5}$]
Let $1<p\leq\alpha<s\leq\infty$ and $f\in(L^p,L^s)^{\alpha}(\nu,w;\mu)$ with $w\in\Delta_2$ and $\mu\in\Delta_2$. For arbitrary $y\in\mathbb R^n$, set $Q=Q(y,\ell)$ for the cube centered at $y$ and of the side length $\ell$. Let
\begin{equation*}
f=f\cdot\chi_{2Q}+f\cdot\chi_{(2Q)^c}:=f_1+f_2,
\end{equation*}
where $\chi_{2Q}$ denotes the characteristic function of $2Q=Q(y,2\ell)$. Then for given $y\in\mathbb R^n$ and $\ell>0$, we write
\begin{align}\label{K}
&w(Q(y,\ell))^{1/{\alpha}-1/p-1/s}\big\|I_{\gamma}(f)\cdot\chi_{Q(y,\ell)}\big\|_{WL^p(w)}\notag\\
&\leq 2\cdot w(Q(y,\ell))^{1/{\alpha}-1/p-1/s}\big\|I_{\gamma}(f_1)\cdot\chi_{Q(y,\ell)}\big\|_{WL^p(w)}\notag\\
&+2\cdot w(Q(y,\ell))^{1/{\alpha}-1/p-1/s}\big\|I_{\gamma}(f_2)\cdot\chi_{Q(y,\ell)}\big\|_{WL^p(w)}\notag\\
&:=K_1(y,\ell)+K_2(y,\ell).
\end{align}
Using Theorem \ref{Two1}, we get
\begin{align}\label{K1}
K_1(y,\ell)&\leq 2\cdot w(Q(y,\ell))^{1/{\alpha}-1/p-1/s}\big\|I_\gamma(f_1)\big\|_{WL^p(w)}\notag\\
&\leq C\cdot w(Q(y,\ell))^{1/{\alpha}-1/p-1/s}
\bigg(\int_{Q(y,2\ell)}|f(x)|^p\nu(x)\,dx\bigg)^{1/p}\notag\\
&=C\cdot w(Q(y,2\ell))^{1/{\alpha}-1/p-1/s}\big\|f\cdot\chi_{Q(y,2\ell)}\big\|_{L^p(\nu)}\notag\\
&\times \frac{w(Q(y,\ell))^{1/{\alpha}-1/p-1/s}}{w(Q(y,2\ell))^{1/{\alpha}-1/p-1/s}}.
\end{align}
Moreover, since $1/{\alpha}-1/p-1/s<0$ and $w\in \Delta_2$, then by doubling inequality \eqref{weights}(consider cube $Q$ instead of ball $B$), we obtain
\begin{equation}\label{doubling3}
\frac{w(Q(y,\ell))^{1/{\alpha}-1/p-1/s}}{w(Q(y,2\ell))^{1/{\alpha}-1/p-1/s}}\leq C.
\end{equation}
Substituting the above inequality \eqref{doubling3} into \eqref{K1}, we thus obtain
\begin{equation}\label{k1yr}
K_1(y,\ell)\leq C\cdot w(Q(y,2\ell))^{1/{\alpha}-1/p-1/s}\big\|f\cdot\chi_{Q(y,2\ell)}\big\|_{L^p(\nu)}.
\end{equation}
We now estimate the second term $K_2(y,\ell)$. Using the same methods and steps as we deal with $I_2(y,r)$ in Theorem \ref{mainthm:1}, we can also obtain that for any $x\in Q(y,\ell)$,
\begin{equation}\label{alpha1}
\big|I_{\gamma}(f_2)(x)\big|\leq C\sum_{j=1}^\infty\frac{1}{|Q(y,2^{j+1}\ell)|^{1-{\gamma}/n}}\int_{Q(y,2^{j+1}\ell)}|f(z)|\,dz.
\end{equation}
This pointwise estimate \eqref{alpha1} together with Chebyshev's inequality implies
\begin{equation*}
\begin{split}
K_2(y,\ell)&\leq 2\cdot w(Q(y,\ell))^{1/{\alpha}-1/p-1/s}\bigg(\int_{Q(y,\ell)}\big|I_{\gamma}(f_2)(x)\big|^pw(x)\,dx\bigg)^{1/p}\\
&\leq C\cdot w(Q(y,\ell))^{1/{\alpha}-1/s}
\sum_{j=1}^\infty\frac{1}{|Q(y,2^{j+1}\ell)|^{1-{\gamma}/n}}\int_{Q(y,2^{j+1}\ell)}|f(z)|\,dz.
\end{split}
\end{equation*}
A further application of H\"older's inequality yields
\begin{equation*}
\begin{split}
K_2(y,\ell)&\leq C\cdot w(Q(y,\ell))^{1/{\alpha}-1/s}
\sum_{j=1}^\infty\frac{1}{|Q(y,2^{j+1}\ell)|^{1-{\gamma}/n}}\bigg(\int_{Q(y,2^{j+1}\ell)}|f(z)|^p\nu(z)\,dz\bigg)^{1/p}\\
&\times\bigg(\int_{Q(y,2^{j+1}\ell)}\nu(z)^{-p'/p}\,dz\bigg)^{1/{p'}}\\
&=C\sum_{j=1}^\infty w\big(Q(y,2^{j+1}\ell)\big)^{1/{\alpha}-1/p-1/s}\big\|f\cdot\chi_{Q(y,2^{j+1}\ell)}\big\|_{L^p(\nu)}\\
&\times\frac{w(Q(y,\ell))^{1/{\alpha}-1/s}}{w(Q(y,2^{j+1}\ell))^{1/{\alpha}-1/s}}
\cdot\frac{w(Q(y,2^{j+1}\ell))^{1/p}}{|Q(y,2^{j+1}\ell)|^{1-{\gamma}/n}}\bigg(\int_{Q(y,2^{j+1}\ell)}\nu(z)^{-p'/p}\,dz\bigg)^{1/{p'}}.
\end{split}
\end{equation*}
In addition, we apply H\"older's inequality with exponent $r$ to get
\begin{equation}\label{U}
w\big(Q(y,2^{j+1}\ell)\big)=\int_{Q(y,2^{j+1}\ell)}w(z)\,dz
\leq\big|Q(y,2^{j+1}\ell)\big|^{1/{r'}}\bigg(\int_{Q(y,2^{j+1}\ell)}w(z)^r\,dz\bigg)^{1/r}.
\end{equation}
Hence, in view of \eqref{U}, we have
\begin{equation}\label{k2yr}
\begin{split}
K_2(y,\ell)&\leq C\sum_{j=1}^\infty w\big(Q(y,2^{j+1}\ell)\big)^{1/{\alpha}-1/p-1/s}\big\|f\cdot\chi_{Q(y,2^{j+1}\ell)}\big\|_{L^p(\nu)}
\cdot\frac{w(Q(y,\ell))^{1/{\alpha}-1/s}}{w(Q(y,2^{j+1}\ell))^{1/{\alpha}-1/s}}\\
&\times\frac{|Q(y,2^{j+1}\ell)|^{1/{(r'p)}}}{|Q(y,2^{j+1}\ell)|^{1-{\gamma}/n}}
\bigg(\int_{Q(y,2^{j+1}\ell)}w(z)^r\,dz\bigg)^{1/{(rp)}}
\bigg(\int_{Q(y,2^{j+1}\ell)}\nu(z)^{-p'/p}\,dz\bigg)^{1/{p'}}\\
&\leq C\sum_{j=1}^\infty w\big(Q(y,2^{j+1}\ell)\big)^{1/{\alpha}-1/p-1/s}\big\|f\cdot\chi_{Q(y,2^{j+1}\ell)}\big\|_{L^p(\nu)}
\cdot\frac{w(Q(y,\ell))^{1/{\alpha}-1/s}}{w(Q(y,2^{j+1}\ell))^{1/{\alpha}-1/s}}.
\end{split}
\end{equation}
The last inequality is obtained by the $A_p$-type condition \eqref{assump1.1} on $(w,\nu)$. Furthermore, since $w\in\Delta_2$, we can easily check that there exists a \emph{reverse doubling} constant $D=D(w)>1$ independent of $Q$ such that (see Lemma 4.1 in \cite{komori})
\begin{equation*}
w(2Q)\geq D\cdot w(Q), \quad \mbox{for any cube }\,Q\subset\mathbb R^n,
\end{equation*}
which implies that for any positive integer $j\in\mathbb Z^+$, $w(2^{j+1}Q)\geq D^{j+1}\cdot w(Q)$ by iteration. Hence,
\begin{align}\label{5}
\sum_{j=1}^\infty\frac{w(Q(y,\ell))^{1/{\alpha}-1/s}}{w(Q(y,2^{j+1}\ell))^{1/{\alpha}-1/s}}
&\leq \sum_{j=1}^\infty\left(\frac{w(Q(y,\ell))}{D^{j+1}\cdot w(Q(y,\ell))}\right)^{1/{\alpha}-1/s}\notag\\
&=\sum_{j=1}^\infty\left(\frac{1}{D^{j+1}}\right)^{1/{\alpha}-1/s}\notag\\
&\leq C,
\end{align}
where the last series is convergent since the \emph{reverse doubling} constant $D>1$ and $1/{\alpha}-1/s>0$. Therefore by taking the $L^s(\mu)$-norm of both sides of \eqref{K}(with respect to the variable $y$), and then using Minkowski's inequality, \eqref{k1yr}, \eqref{k2yr} and \eqref{5}, we have
\begin{equation*}
\begin{split}
&\Big\|w(Q(y,\ell))^{1/{\alpha}-1/p-1/s}\big\|I_{\gamma}(f)\cdot\chi_{Q(y,\ell)}\big\|_{WL^p(w)}\Big\|_{L^s(\mu)}\\
&\leq\big\|K_1(y,\ell)\big\|_{L^s(\mu)}+\big\|K_2(y,\ell)\big\|_{L^s(\mu)}\\
&\leq C\Big\|w(Q(y,2\ell))^{1/{\alpha}-1/p-1/s}\big\|f\cdot\chi_{Q(y,2\ell)}\big\|_{L^p(\nu)}\Big\|_{L^s(\mu)}\\
&+C\sum_{j=1}^\infty\Big\|w(Q(y,2^{j+1}\ell))^{1/{\alpha}-1/p-1/s}\big\|f\cdot\chi_{Q(y,2^{j+1}\ell)}\big\|_{L^p(\nu)}\Big\|_{L^s(\mu)}
\times\frac{w(Q(y,\ell))^{1/{\alpha}-1/s}}{w(Q(y,2^{j+1}\ell))^{1/{\alpha}-1/s}}\\
&\leq C\big\|f\big\|_{(L^p,L^s)^{\alpha}(\nu,w;\mu)}+C\big\|f\big\|_{(L^p,L^s)^{\alpha}(\nu,w;\mu)}
\times\sum_{j=1}^\infty\frac{w(Q(y,\ell))^{1/{\alpha}-1/s}}{w(Q(y,2^{j+1}\ell))^{1/{\alpha}-1/s}}\\
&\leq C\big\|f\big\|_{(L^p,L^s)^{\alpha}(\nu,w;\mu)}.
\end{split}
\end{equation*}
Finally, by taking the supremum over all $\ell>0$, we finish the proof of Theorem \ref{mainthm:5}.
\end{proof}

\begin{proof}[Proof of Theorem $\ref{mainthm:6}$]
Let $1<p\leq\alpha<s\leq\infty$ and $f\in(L^p,L^s)^{\alpha}(\nu,w;\mu)$ with $w\in A_\infty$ and $\mu\in\Delta_2$. For an arbitrary cube $Q=Q(y,\ell)$ in $\mathbb R^n$, as before, we set
\begin{equation*}
f=f_1+f_2,\qquad f_1=f\cdot\chi_{2Q},\quad  f_2=f\cdot\chi_{(2Q)^c}.
\end{equation*}
Then for given $y\in\mathbb R^n$ and $\ell>0$, we write
\begin{align}\label{Kprime}
&w(Q(y,\ell))^{1/{\alpha}-1/p-1/s}\big\|[b,I_{\gamma}](f)\cdot\chi_{Q(y,\ell)}\big\|_{WL^p(w)}\notag\\
&\leq 2\cdot w(Q(y,\ell))^{1/{\alpha}-1/p-1/s}\big\|[b,I_{\gamma}](f_1)\cdot\chi_{Q(y,\ell)}\big\|_{WL^p(w)}\notag\\
&+2\cdot w(Q(y,\ell))^{1/{\alpha}-1/p-1/s}\big\|[b,I_{\gamma}](f_2)\cdot\chi_{Q(y,\ell)}\big\|_{WL^p(w)}\notag\\
&:=K'_1(y,\ell)+K'_2(y,\ell).
\end{align}
Since $w\in A_\infty$, we know that $w\in\Delta_2$. Applying Theorem \ref{Two2} and inequality (\ref{doubling3}), we get
\begin{align}\label{K1prime}
K'_1(y,\ell)&\leq 2\cdot w(Q(y,\ell))^{1/{\alpha}-1/p-1/s}\big\|[b,I_{\gamma}](f_1)\big\|_{WL^p(w)}\notag\\
&\leq C\cdot w(Q(y,\ell))^{1/{\alpha}-1/p-1/s}
\bigg(\int_{Q(y,2\ell)}|f(x)|^p\nu(x)\,dx\bigg)^{1/p}\notag\\
&=C\cdot w(Q(y,2\ell))^{1/{\alpha}-1/p-1/s}\big\|f\cdot\chi_{Q(y,2\ell)}\big\|_{L^p(\nu)}\notag\\
&\times \frac{w(Q(y,\ell))^{1/{\alpha}-1/p-1/s}}{w(Q(y,2\ell))^{1/{\alpha}-1/p-1/s}}\notag\\
&\leq C\cdot w(Q(y,2\ell))^{1/{\alpha}-1/p-1/s}\big\|f\cdot\chi_{Q(y,2\ell)}\big\|_{L^p(\nu)}.
\end{align}
Next we estimate the other term $K'_2(y,\ell)$. For any $x\in Q(y,\ell)$, from the definition of $[b,I_{\gamma}]$, one can see that
\begin{equation*}
\begin{split}
\big|[b,I_{\gamma}](f_2)(x)\big|
&\leq \big|b(x)-b_{Q(y,\ell)}\big|\cdot\big|I_\gamma(f_2)(x)\big|
+\Big|I_\gamma\big([b_{Q(y,\ell)}-b]f_2\big)(x)\Big|\\
&:=\xi(x)+\eta(x).
\end{split}
\end{equation*}
Consequently, we can further divide $K'_2(y,\ell)$ into two parts:
\begin{equation*}
\begin{split}
K'_2(y,\ell)\leq&4\cdot w(Q(y,\ell))^{1/{\alpha}-1/p-1/s}\big\|\xi(\cdot)\cdot\chi_{Q(y,\ell)}\big\|_{WL^p(w)}\\
&+4\cdot w(Q(y,\ell))^{1/{\alpha}-1/p-1/s}\big\|\eta(\cdot)\cdot\chi_{Q(y,\ell)}\big\|_{WL^p(w)}\\
:=&K'_3(y,\ell)+K'_4(y,\ell).
\end{split}
\end{equation*}
For the term $K'_3(y,\ell)$, it follows directly from Chebyshev's inequality and estimate \eqref{alpha1} that
\begin{equation*}
\begin{split}
K'_3(y,\ell)&\leq4\cdot w(Q(y,\ell))^{1/{\alpha}-1/p-1/s}\bigg(\int_{Q(y,\ell)}\big|\xi(x)\big|^pw(x)\,dx\bigg)^{1/p}\\
&\leq C\cdot w(Q(y,\ell))^{1/{\alpha}-1/p-1/s}\bigg(\int_{Q(y,\ell)}\big|b(x)-b_{Q(y,\ell)}\big|^pw(x)\,dx\bigg)^{1/p}\\
&\times\sum_{j=1}^\infty\frac{1}{|Q(y,2^{j+1}\ell)|^{1-\gamma/n}}\int_{Q(y,2^{j+1}\ell)}|f(z)|\,dz\\
&\leq C\cdot w(Q(y,\ell))^{1/{\alpha}-1/s}
\sum_{j=1}^\infty\frac{1}{|Q(y,2^{j+1}\ell)|^{1-\gamma/n}}\int_{Q(y,2^{j+1}\ell)}|f(z)|\,dz,
\end{split}
\end{equation*}
where in the last inequality we have used the fact that Lemma \ref{BMO}$(ii)$ still holds with ball $B$ replaced by cube $Q$, when $w$ is an $A_{\infty}$ weight. Arguing as in the proof of Theorem \ref{mainthm:5}, we can also obtain that
\begin{equation*}
K'_3(y,\ell)\leq C\sum_{j=1}^\infty w\big(Q(y,2^{j+1}\ell)\big)^{1/{\alpha}-1/p-1/s}\big\|f\cdot\chi_{Q(y,2^{j+1}\ell)}\big\|_{L^p(\nu)}
\cdot\frac{w(Q(y,\ell))^{1/{\alpha}-1/s}}{w(Q(y,2^{j+1}\ell))^{1/{\alpha}-1/s}}.
\end{equation*}
Let us now estimate the term $K'_4(y,\ell)$. Using the same methods and steps as we deal with $J_2(y,r)$ in Theorem \ref{mainthm:3}, we can show the following pointwise estimate as well.
\begin{equation*}
\begin{split}
\eta(x)&=\Big|I_\gamma\big([b_{Q(y,\ell)}-b]f_2\big)(x)\Big|\\
&\leq C\sum_{j=1}^\infty\frac{1}{|Q(y,2^{j+1}\ell)|^{1-\gamma/n}}\int_{Q(y,2^{j+1}\ell)}\big|b(z)-b_{Q(y,\ell)}\big|\cdot|f(z)|\,dz.
\end{split}
\end{equation*}
This, together with Chebyshev's inequality implies
\begin{equation*}
\begin{split}
K'_4(y,\ell)&\leq4\cdot w(Q(y,\ell))^{1/{\alpha}-1/p-1/s}\bigg(\int_{Q(y,\ell)}\big|\eta(x)\big|^pw(x)\,dx\bigg)^{1/p}\\
&\leq C\cdot w(Q(y,\ell))^{1/{\alpha}-1/s}
\sum_{j=1}^\infty\frac{1}{|Q(y,2^{j+1}\ell)|^{1-\gamma/n}}\int_{Q(y,2^{j+1}\ell)}\big|b(z)-b_{Q(y,\ell)}\big|\cdot|f(z)|\,dz\\
&\leq C\cdot w(Q(y,\ell))^{1/{\alpha}-1/s}
\sum_{j=1}^\infty\frac{1}{|Q(y,2^{j+1}\ell)|^{1-\gamma/n}}\int_{Q(y,2^{j+1}\ell)}\big|b(z)-b_{Q(y,2^{j+1}\ell)}\big|\cdot|f(z)|\,dz\\
&+C\cdot w(Q(y,\ell))^{1/{\alpha}-1/s}
\sum_{j=1}^\infty\frac{1}{|Q(y,2^{j+1}\ell)|^{1-\gamma/n}}\int_{Q(y,2^{j+1}\ell)}\big|b_{Q(y,2^{j+1}\ell)}-b_{Q(y,\ell)}\big|\cdot|f(z)|\,dz\\
&:=K'_5(y,\ell)+K'_6(y,\ell).
\end{split}
\end{equation*}
An application of H\"older's inequality leads to that
\begin{equation*}
\begin{split}
K'_5(y,\ell)&\leq C\cdot w(Q(y,\ell))^{1/{\alpha}-1/s}
\sum_{j=1}^\infty\frac{1}{|Q(y,2^{j+1}\ell)|^{1-\gamma/n}}\bigg(\int_{Q(y,2^{j+1}\ell)}|f(z)|^p\nu(z)\,dz\bigg)^{1/p}\\
&\times\bigg(\int_{Q(y,2^{j+1}\ell)}\big|b(z)-b_{Q(y,2^{j+1}\ell)}\big|^{p'}\nu(z)^{-p'/p}\,dz\bigg)^{1/{p'}}\\
&=C\cdot w(Q(y,\ell))^{1/{\alpha}-1/s}\sum_{j=1}^\infty\frac{\big\|f\cdot\chi_{Q(y,2^{j+1}\ell)}\big\|_{L^p(\nu)}}{|Q(y,2^{j+1}\ell)|^{1-\gamma/n}}\\
&\times\big|Q(y,2^{j+1}\ell)\big|^{1/{p'}}\Big\|\big[b-b_{Q(y,2^{j+1}\ell)}\big]\cdot\nu^{-1/p}\Big\|_{\mathcal C,Q(y,2^{j+1}\ell)},
\end{split}
\end{equation*}
where $\mathcal C(t)=t^{p'}$ is a Young function. For $1<p<\infty$, it is easy to see that the inverse function of $\mathcal C(t)$ is $\mathcal C^{-1}(t)=t^{1/{p'}}$. Also observe that the following equality holds:
\begin{equation*}
\begin{split}
\mathcal C^{-1}(t)&=t^{1/{p'}}\\
&=\frac{t^{1/{p'}}}{1+\log^+ t}\times\big(1+\log^+t\big)\\
&=\mathcal A^{-1}(t)\cdot\mathcal B^{-1}(t),
\end{split}
\end{equation*}
where
\begin{equation*}
\mathcal A(t)\approx t^{p'}\cdot(1+\log^+t)^{p'},\qquad \mbox{and}\qquad \mathcal B(t)\approx \exp(t)-1.
\end{equation*}
Thus, by generalized H\"older's inequality \eqref{three} and estimate \eqref{Jensen}(consider cube $Q$ instead of ball $B$ when $w\equiv1$), we have
\begin{equation*}
\begin{split}
\Big\|\big[b-b_{Q(y,2^{j+1}\ell)}\big]\cdot\nu^{-1/p}\Big\|_{\mathcal C,Q(y,2^{j+1}\ell)}
&\leq C\Big\|b-b_{Q(y,2^{j+1}\ell)}\Big\|_{\mathcal B,Q(y,2^{j+1}\ell)}\cdot\Big\|\nu^{-1/p}\Big\|_{\mathcal A,Q(y,2^{j+1}\ell)}\\
&\leq C\|b\|_*\cdot\Big\|\nu^{-1/p}\Big\|_{\mathcal A,Q(y,2^{j+1}\ell)}.
\end{split}
\end{equation*}
Moreover, in view of \eqref{U}, we can deduce that
\begin{equation*}
\begin{split}
K'_5(y,\ell)&\leq C\|b\|_*\cdot w(Q(y,\ell))^{1/{\alpha}-1/s}
\sum_{j=1}^\infty\frac{\big\|f\cdot\chi_{Q(y,2^{j+1}\ell)}\big\|_{L^p(\nu)}}{|Q(y,2^{j+1}\ell)|^{1/p-\gamma/n}}
\cdot\Big\|\nu^{-1/p}\Big\|_{\mathcal A,Q(y,2^{j+1}\ell)}\\
&=C\|b\|_*\sum_{j=1}^\infty w\big(Q(y,2^{j+1}\ell)\big)^{1/{\alpha}-1/p-1/s}\big\|f\cdot\chi_{Q(y,2^{j+1}\ell)}\big\|_{L^p(\nu)}
\cdot\frac{w(Q(y,\ell))^{1/{\alpha}-1/s}}{w(Q(y,2^{j+1}\ell))^{1/{\alpha}-1/s}}\\
&\times\frac{w(Q(y,2^{j+1}\ell))^{1/p}}{|Q(y,2^{j+1}\ell)|^{1/p-\gamma/n}}
\cdot\Big\|\nu^{-1/p}\Big\|_{\mathcal A,Q(y,2^{j+1}\ell)}\\
&\leq C\|b\|_*\sum_{j=1}^\infty w\big(Q(y,2^{j+1}\ell)\big)^{1/{\alpha}-1/p-1/s}\big\|f\cdot\chi_{Q(y,2^{j+1}\ell)}\big\|_{L^p(\nu)}
\cdot\frac{w(Q(y,\ell))^{1/{\alpha}-1/s}}{w(Q(y,2^{j+1}\ell))^{1/{\alpha}-1/s}}\\
&\times\frac{|Q(y,2^{j+1}\ell)|^{1/{(r'p)}}}{|Q(y,2^{j+1}\ell)|^{1/p-\gamma/n}}
\bigg(\int_{Q(y,2^{j+1}\ell)}w(z)^r\,dz\bigg)^{1/{(rp)}}\cdot\Big\|\nu^{-1/p}\Big\|_{\mathcal A,Q(y,2^{j+1}\ell)}\\
&\leq C\|b\|_*\sum_{j=1}^\infty w\big(Q(y,2^{j+1}\ell)\big)^{1/{\alpha}-1/p-1/s}\big\|f\cdot\chi_{Q(y,2^{j+1}\ell)}\big\|_{L^p(\nu)}
\cdot\frac{w(Q(y,\ell))^{1/{\alpha}-1/s}}{w(Q(y,2^{j+1}\ell))^{1/{\alpha}-1/s}}.
\end{split}
\end{equation*}
The last inequality is obtained by the $A_p$-type condition \eqref{assump2.1} on $(w,\nu)$. It remains to estimate the last term $K'_6(y,\ell)$.
 Applying Lemma \ref{BMO}$(i)$(use $Q$ instead of $B$) and H\"older's inequality, we get
\begin{equation*}
\begin{split}
K'_6(y,\ell)&\leq C\cdot w(Q(y,\ell))^{1/{\alpha}-1/s}
\sum_{j=1}^\infty\frac{(j+1)\|b\|_*}{|Q(y,2^{j+1}\ell)|^{1-\gamma/n}}\int_{Q(y,2^{j+1}\ell)}|f(z)|\,dz\\
&\leq C\cdot w(Q(y,\ell))^{1/{\alpha}-1/s}
\sum_{j=1}^\infty\frac{(j+1)\|b\|_*}{|Q(y,2^{j+1}\ell)|^{1-\gamma/n}}\bigg(\int_{Q(y,2^{j+1}\ell)}|f(z)|^p\nu(z)\,dz\bigg)^{1/p}\\
&\times\bigg(\int_{Q(y,2^{j+1}\ell)}\nu(z)^{-p'/p}\,dz\bigg)^{1/{p'}}\\
&=C\|b\|_*\sum_{j=1}^\infty w\big(Q(y,2^{j+1}\ell)\big)^{1/{\alpha}-1/p-1/s}\big\|f\cdot\chi_{Q(y,2^{j+1}\ell)}\big\|_{L^p(\nu)}\\
&\times\big(j+1\big)\cdot\frac{w(Q(y,\ell))^{1/{\alpha}-1/s}}{w(Q(y,2^{j+1}\ell))^{1/{\alpha}-1/s}}
\cdot\frac{w(Q(y,2^{j+1}\ell))^{1/p}}{|Q(y,2^{j+1}\ell)|^{1-\gamma/n}}\bigg(\int_{Q(y,2^{j+1}\ell)}\nu(z)^{-p'/p}\,dz\bigg)^{1/{p'}}.
\end{split}
\end{equation*}
Let $\mathcal C(t)$ and $\mathcal A(t)$ be the same as before. Obviously, $\mathcal C(t)\leq\mathcal A(t)$ for all $t>0$, then it is not difficult to see that for any given cube $Q$ in $\mathbb R^n$, we have $\big\|f\big\|_{\mathcal C,Q}\leq\big\|f\big\|_{\mathcal A,Q}$ by definition, which implies that condition \eqref{assump2.1} is stronger that condition \eqref{assump1.1}. This fact together with \eqref{U} yields
\begin{equation*}
\begin{split}
K'_6(y,\ell)&\leq C\|b\|_*\sum_{j=1}^\infty w\big(Q(y,2^{j+1}\ell)\big)^{1/{\alpha}-1/p-1/s}\big\|f\cdot\chi_{Q(y,2^{j+1}\ell)}\big\|_{L^p(\nu)}\\
&\times\big(j+1\big)\cdot\frac{w(Q(y,\ell))^{1/{\alpha}-1/s}}{w(Q(y,2^{j+1}\ell))^{1/{\alpha}-1/s}}\\
&\times\frac{|Q(y,2^{j+1}\ell)|^{1/{(r'p)}}}{|Q(y,2^{j+1}\ell)|^{1-\gamma/n}}
\bigg(\int_{Q(y,2^{j+1}\ell)}w(z)^r\,dz\bigg)^{1/{(rp)}}
\bigg(\int_{Q(y,2^{j+1}\ell)}\nu(z)^{-p'/p}\,dz\bigg)^{1/{p'}}\\
&\leq C\|b\|_*\sum_{j=1}^\infty w\big(Q(y,2^{j+1}\ell)\big)^{1/{\alpha}-1/p-1/s}\big\|f\cdot\chi_{Q(y,2^{j+1}\ell)}\big\|_{L^p(\nu)}\\
&\times\big(j+1\big)\cdot\frac{w(Q(y,\ell))^{1/{\alpha}-1/s}}{w(Q(y,2^{j+1}\ell))^{1/{\alpha}-1/s}}.
\end{split}
\end{equation*}
Summing up all the above estimates, we conclude that
\begin{align}\label{K2prime}
K'_2(y,\ell)&\leq C\sum_{j=1}^\infty w\big(Q(y,2^{j+1}\ell)\big)^{1/{\alpha}-1/p-1/s}\big\|f\cdot\chi_{Q(y,2^{j+1}\ell)}\big\|_{L^p(\nu)}\notag\\
&\times\big(j+1\big)\cdot\frac{w(Q(y,\ell))^{1/{\alpha}-1/s}}{w(Q(y,2^{j+1}\ell))^{1/{\alpha}-1/s}}.
\end{align}
Moreover, since $w$ is an $A_{\infty}$ weight, one has $w\in\Delta_2$. Then there exists a \emph{reverse doubling} constant $D=D(w)>1$ such that for any positive integer $j\in\mathbb Z^+$, $w(Q(y,2^{j+1}\ell))\geq D^{j+1}\cdot w(Q(y,\ell))$. This allows us to get the following:
\begin{align}\label{6}
\sum_{j=1}^\infty\big(j+1\big)\cdot\frac{w(Q(y,\ell))^{1/{\alpha}-1/s}}{w(Q(y,2^{j+1}\ell))^{1/{\alpha}-1/s}}
&\leq\sum_{j=1}^\infty(j+1)\cdot\left(\frac{w(Q(y,\ell))}{D^{j+1}\cdot w(Q(y,\ell))}\right)^{1/{\alpha}-1/s}\notag\\
&=\sum_{j=1}^\infty(j+1)\cdot\left(\frac{1}{D^{j+1}}\right)^{1/{\alpha}-1/s}\notag\\
&\leq C.
\end{align}
Notice that the exponent $(1/{\alpha}-1/s)$ is positive because $\alpha<s$, which guarantees that the last series is convergent.
Thus by taking the $L^s(\mu)$-norm of both sides of \eqref{Kprime}(with respect to the variable $y$), and then using Minkowski's inequality, \eqref{K1prime}, \eqref{K2prime} and \eqref{6}, we finally obtain
\begin{equation*}
\begin{split}
&\Big\|w(Q(y,\ell))^{1/{\alpha}-1/p-1/s}\big\|[b,I_{\gamma}](f)\cdot\chi_{Q(y,\ell)}\big\|_{WL^p(w)}\Big\|_{L^s(\mu)}\\
&\leq\big\|K'_1(y,\ell)\big\|_{L^s(\mu)}+\big\|K'_2(y,\ell)\big\|_{L^s(\mu)}\\
&\leq C\Big\|w(Q(y,2\ell))^{1/{\alpha}-1/p-1/s}\big\|f\cdot\chi_{Q(y,2\ell)}\big\|_{L^p(\nu)}\Big\|_{L^s(\mu)}\\
&+C\sum_{j=1}^\infty\Big\|w(Q(y,2^{j+1}\ell))^{1/{\alpha}-1/p-1/s}\big\|f\cdot\chi_{Q(y,2^{j+1}\ell)}\big\|_{L^p(\nu)}\Big\|_{L^s(\mu)}\\
\end{split}
\end{equation*}
\begin{equation*}
\begin{split}
&\times\big(j+1\big)\cdot\frac{w(Q(y,\ell))^{1/{\alpha}-1/s}}{w(Q(y,2^{j+1}\ell))^{1/{\alpha}-1/s}}\\
&\leq C\big\|f\big\|_{(L^p,L^s)^{\alpha}(\nu,w;\mu)}+C\big\|f\big\|_{(L^p,L^s)^{\alpha}(\nu,w;\mu)}
\times\sum_{j=1}^\infty\big(j+1\big)\cdot\frac{w(Q(y,\ell))^{1/{\alpha}-1/s}}{w(Q(y,2^{j+1}\ell))^{1/{\alpha}-1/s}}\\
&\leq C\big\|f\big\|_{(L^p,L^s)^{\alpha}(\nu,w;\mu)}.
\end{split}
\end{equation*}
We therefore conclude the proof of Theorem \ref{mainthm:6} by taking the supremum over all $\ell>0$.
\end{proof}
In view of \eqref{dominate2} and \eqref{dominate1}, as an immediate consequence of Theorem \ref{mainthm:5}, we have the following result.
\begin{corollary}
Let $0<\gamma<n$, $1<p\leq\alpha<s\leq\infty$ and $\mu\in\Delta_2$. Given a pair of weights $(w,\nu)$, suppose that for some $r>1$ and for all cubes $Q$ in $\mathbb R^n$, \eqref{assump1.1} holds.
If $w\in \Delta_2$, then both fractional maximal operator $M_{\gamma}$ and generalized fractional integrals $\mathcal L^{-\gamma/2}$ are bounded from $(L^p,L^s)^{\alpha}(\nu,w;\mu)$ into $(WL^p,L^s)^{\alpha}(w;\mu)$.
\end{corollary}

\end{document}